\documentclass[12pt,a4paper]{article}
\usepackage{amsmath,amstext,amssymb,amscd, color}
\usepackage{graphicx}

\usepackage[english]{babel}

\newcommand{\Xcomment}[1]{}

\newtheorem{theorem}{Theorem}[section]
\newtheorem{lemma}[theorem]{Lemma}
\newtheorem{corollary}[theorem]{Corollary}

\newtheorem{prop}[theorem]{Proposition}

\newcommand{\SEC}[1]{\ref{sec:#1}}  
\newcommand{\SSEC}[1]{\ref{ssec:#1}}  

\makeatletter \@addtoreset{equation}{section} \makeatother

\newenvironment{proof}{\noindent{\bf Proof}~}%
{\hfill$\qed$\medskip}

\def\qed{ \ \vrule width.1cm height.3cm depth0cm}

\newenvironment{numitem1}{\refstepcounter{equation}\begin{enumerate}%
\item[(\thesection.\arabic{equation})]}{\end{enumerate}}

\newcommand{\refeq}[1]{(\ref{eq:#1})}  

 \makeatletter
\renewcommand{\section}{\@startsection{section}{1}{0pt}%
{-3.5ex plus -1ex minus -.2ex}{2.3ex plus .2ex}%
{\normalfont\Large}}
 \makeatother

 \makeatletter
\renewcommand{\subsection}{\@startsection{subsection}{2}{0pt}%
{-3.0ex plus -1ex minus -.2ex}{-1.5ex plus .2ex}%
{\normalfont\normalsize\bf}}
 \makeatother

  \Xcomment{
 \makeatletter
\renewcommand{\subsection}{\@startsection{subsection}{2}{0pt}%
{-3.0ex plus -1ex minus -.2ex}{1.5ex plus .2ex}%
{\normalfont\normalsize\bf}}
 \makeatother
 }

\def\Rset{{\mathbb R}}
\def\Zset{{\mathbb Z}}

\def\Ascr{{\cal A}}
\def\Bscr{{\cal B}}
\def\Cscr{{\cal C}}
\def\Dscr{{\cal D}}

\def\Iscr{{\cal I}}

\def\Mscr{{\cal M}}
\def\Nscr{{\cal N}}

\def\Pscr{{\cal P}}

\def\Rscr{{\cal R}}
\def\Sscr{{\cal S}}
\def\Tscr{{\cal T}}

\def\Wscr{{\cal W}}

\def\tilde{\widetilde}
\def\hat{\widehat}
\def\bar{\overline}
\def\eps{\epsilon}

\def\bfC{{\bf C}}

\def\bfS{{\bf S}}

\def\bfW{{\bf W}}

\def\Mscrw{{\Mscr^{\rm w}}}
\def\Mscre{{\Mscr^{\rm e}}}

\def\Zfr{Z^{\,\rm fr}}
\def\Zrear{Z^{\,\rm rear}}
\def\Zrim{Z^{\,\rm rim}}
\def\Cfr{C^{\,\rm fr}}
\def\Crear{C^{\,\rm rear}}
\def\Crim{C^{\,\rm rim}}

\def\Deltafr{\Delta^{\,\rm fr}}
\def\Deltarear{\Delta^{\,\rm rear}}

\def\epsfr{{\eps,\rm fr}}
\def\epsrear{{\eps,\rm rear}}

\def\Nscrup{\Nscr^{\uparrow}}
\def\Nscrdown{\Nscr^{\downarrow}}

\def\Qfrag{Q^\equiv}
\def\Cfrag{C^\equiv}

\def\Qfragen{{Q^\equiv_{\rm en}}}
\def\precen{\prec_{\rm en}}

\begin{document}

 \title{The weak separation in higher dimensions}

 \author{Vladimir I.~Danilov
\thanks{Central Institute of Economics and
Mathematics of the RAS, 47, Nakhimovskii Prospect, 117418 Moscow, Russia;
email: danilov@cemi.rssi.ru.}
 \and
Alexander V.~Karzanov
\thanks{Central Institute of Economics and Mathematics of
the RAS, 47, Nakhimovskii Prospect, 117418 Moscow, Russia; email:
akarzanov7@gmail.com. Corresponding author. }
  \and
Gleb A.~Koshevoy
\thanks{The Institute for Information Transmission Problems of
the RAS, 19, Bol'shoi Karetnyi per., 127051 Moscow, Russia; email:
koshevoyga@gmail.com. Supported in part by grant RSF 16-11-10075. }
  }

\date{}

 \maketitle

 \begin{quote}
 {\bf Abstract.} \small
For an odd integer $r>0$ and an integer $n>r$, we introduce a notion of
\emph{weakly $r$-separated} collections of subsets of $[n]=\{1,2,\ldots,n\}$.
When $r=1$, this corresponds to the concept of weak separation introduced by
Leclerc and Zelevinsky. In this paper, extending results due to
Leclerc-Zelevinsky, we develop a geometric approach to establish a number of
nice combinatorial properties of maximal weakly r-separated collections.

As a supplement, we also discuss an analogous concept when $r$ is even.

Posted on \textsl{arXiv}:1904.09798[math.CO].
 \smallskip


  \end{quote}


\section{Introduction}  \label{sec:intr}

Let $n$ be a positive integer and let $[n]$ denote the set $\{1,2,\ldots,n\}$.
For subsets $X,Y\subseteq[n]$, we write $X<Y$ if the maximal element $\max(X)$
of $X$ is smaller than the minimal element $\min(Y)$ of $Y$, letting
$\max(\emptyset):=0$ and $\min(\emptyset):=n+1$. An \emph{interval} in $[n]$
is a nonempty subset $\{a,a+1,\ldots,b\}$ in it, denoted as $[a,b]$ (so
$[n]=[1,n]$).

The well-known concept of \emph{strongly separated} sets introduced by Leclerc
and Zelevinsky~\cite{LZ} is extended as follows.
  \smallskip

\noindent\textbf{Definition.} ~For $r\in\Zset_{\ge 0}$, sets $A,B\subseteq [n]$
are called (strongly) $r$-\emph{separated} if there is no sequence
$i_1<i_2<\cdots<i_{r+2}$ of elements of $[n]$ such that the elements with odd
indices (namely, $i_1,i_3,\ldots$) belong to one of $A-B$ and $B-A$, while the
elements with even indices ($i_2,i_4,\ldots$) belong to the other of these two
sets (where $A'-B'$ denotes the set difference $\{i\colon A'\ni i\notin B'\}$).
Accordingly, a \emph{set-system} $\Sscr\subseteq 2^{[n]}$ (a collection of
subsets of $[n]$) is called $r$-separated if any two members of $\Sscr$ are
such.
  \smallskip

Equivalently, $A,B\subseteq[n]$ are $r$-separated if there are intervals
$I_1<I_2<\cdots <I_{r'}$ in $[n]$ with $0\le r'\le r+1$ such that one of $A-B$
and $B-A$ is included in $I_1\cup I_3\cup \ldots$\;, and the other in $I_2\cup
I_4\cup\ldots$\;. If, in addition, $r'+|I_1|+\cdots+ |I_{r'}|$ is as small as
possible, we say that $(I_1,\ldots,I_{r'})$ is the \emph{interval cortege}
associated with $A,B$.

In particular, $A,B$ are 0-separated if $A\subseteq B$ or $B\subseteq A$, and
1-separated if either $\max(A-B)<\min(B-A)$ or $\max(B-A)<\min(A-B)$. The
1-separation relation is just what is called the strong separation one
in~\cite{LZ}. The case $r=2$ was studied by Galashin~\cite{gal} (who used the
term ``chord separated'' for 2-separated sets). A study for a general $r$ is
conducted in Galashin and Postnikov~\cite{GP}.

When $A,B$ are $r$-separated but not $(r-1)$-separated, they are called
$(r+1)$-\emph{interlaced}. In other words, the interval cortege associated with
such $A,B$ consists of $r+1$ intervals. For example, $A=\{1,2,5,6,7,10\}$ and
$B=\{2,3,6,9\}$ have the interval cortege $(\{1\},\,\{3\},\,
[5,7],\,\{9\},\{10\})$), and therefore they are 5-interlaced.

Another sort of set separation introduced by Leclerc and Zelevinsky is known
under the name of \emph{weak separation} (which appeared in~\cite{LZ} in
connection with the problem of characterizing quasi-commuting flag minors of a
quantum matrix; for a discussion on this and wider relations between the weak
separation and quantum minors, see also~\cite[Sect.~8]{DK}). We generalize that
notion to ``higher dimensions'' in the following way (where the term ``higher
dimensions'' is justified by appealing to a geometric interpretation, defined
later). When $A,B\subseteq [n]$ are such that $\min(A-B)<\min(B-A)$ and
$\max(A-B)>\max(B-A)$, we say that $A$ \emph{surrounds} $B$.
  \smallskip

\noindent\textbf{Definition.} ~Let $r$ be a positive \emph{odd} integer. Sets
$A,B\subseteq[n]$ are called \emph{weakly $r$-separated} if they are
$r'$-interlaced with $r'\le r+2$, and if $r'=r+2$ takes place, then either (a)
~$A$ surrounds $B$ and $|A|\le |B|$, or (b) ~$B$ surrounds $A$ and $|B|\le|A|$.
Accordingly, a set-system $\Wscr\subseteq 2^{[n]}$ is called weakly
$r$-separated if any two members of $\Wscr$ are such.
 \smallskip

In other words, $A$ and $B$ are weakly $r$-separated if they are either
(strongly) $r$-separated or $(r+2)$-interlaced, and in the latter case, for the
interval cortege $(I_1,\ldots,I_{r+2})$ associated with $A,B$, if the
cardinatilies of $A$ and $B$ are different, say, $|A|<|B|$, then $I_1\cup
I_3\cup\ldots \cup I_{r+2}$ contains $A-B$ (and $I_2\cup I_4\cup\ldots \cup
I_{r+1}$ contains $B-A$). For example, $\{1,2,6\}$ and $\{2,3,4,5\}$ are weakly
1-separated, whereas $\{1,2,5,6,7\}$ and $\{1,3,4,5\}$ are 3-interlaced (having
the interval cortege $(\{2\},[3,4],[6,7])$) but not weakly 1-separated.

When $r=1$, the notion of weak 1-separation turns into the weak separation
of~\cite{LZ}.

In this paper we generalize, to an arbitrary odd $r\ge 1$, two results on
weakly separated collections obtained in~\cite{LZ}. One of those says that
  \begin{numitem1}\label{eq:max_size}
the maximal possible sizes (numbers of members) of strongly and weakly
separated collections in $2^{[n]}$ are the same and equal to $\frac12 n(n+1)+1$
($=\binom{n}{2}+\binom{n}{1}+\binom{n}{0}$).
 \end{numitem1}
To formulate a generalization of~\refeq{max_size}, let $r<n$ and denote the
maximal possible size $|\Sscr|$ of an $r$-separated collection $\Sscr\subseteq
2^{[n]}$ by $s_{n,r}$. Also when $r$ is odd, denote the maximal possible size
of a weakly $r$-separated collection $\Wscr\subseteq 2^{[n]}$ by $w_{n,r}$.
Extending results in~\cite{LZ} (for $r=1$) and~\cite{gal} (for $r=2$), Galashin
and Postnikov~\cite{GP} showed that
  \begin{numitem1} \label{eq:snr}
  ~$s_{n,r}=\binom{n}{\le r+1}$ \qquad $\left(=\binom{n}{r+1}+\binom{n}{r}+\cdots
  +\binom{n}{0}\right)$.
  \end{numitem1}

We prove the following
  \begin{theorem} \label{tm:w=s}
Let $r$ be odd. Then $w_{n,r}=s_{n,r}$.
 \end{theorem}

Another impressive result in~\cite{LZ} says that a weakly separated collection
can be transformed into another one by making a \emph{flip} (a sort of
mutations) ``in the presence of four witnesses''. This relies on the following
property (Theorem~7.1 in~\cite{LZ}):
  \begin{numitem1} \label{eq:four_witn}
let $\Wscr\subset 2^{[n]}$ be weakly separated, and suppose that there are
elements $i<j<k$ of $[n]$ and a set $X\subseteq[n]-\{i,j,k\}$ such that $\Wscr$
contains four sets (``witnesses'') $Xi$, $Xk$, $Xij$, $Xjk$ and a set
$U\in\{Xj,Xik\}$; then the collection obtained from $\Wscr$ by replacing $U$ by
the other member of $\{Xj,Xik\}$ is again weakly separated.
  \end{numitem1}
Hereinafter, for disjoint subsets $A$ and $\{a,\ldots,b\}$ of $[n]$, we use the
abbreviated notation $Aa\ldots b$ for $A\cup\{a,\ldots,b\}$.

We generalize~\refeq{four_witn} as follows.
  \begin{theorem} \label{tm:gen_witn}
For an odd $r$ and $r':=(r+1)/2$, let $P=\{p_1,\ldots,p_{r'}\}$ and
$Q=\{q_0,\ldots,q_{r'}\}$ consist of elements of $[n]$ such that
$q_0<p_1<q_1<p_2< \ldots<p_{r'}<q_{r'}$, and let $X\subseteq[n]-(P\cup Q)$.
Define the set of \emph{neighbors} (or ``witnesses'') of $P,Q$ to be
\begin{equation} \label{eq:neighb}
\Nscr=\Nscr(P,Q):=\{S\subset P\cup Q\,\colon\, S\ne P,Q,\; r'\le |S|\le r'+1\}.
  \end{equation}
Suppose that a weakly $r$-separated collection $\Wscr\subset 2^{[n]}$ contains
a set $U\in\{X\cup P,\,X\cup Q\}$. If, in addition, $\Wscr$ contains the sets
of the form $X\cup S$ for all $S\in\Nscr$, then the collection obtained from
$\Wscr$ by replacing $U$ by the other member $U'$ of $\{X\cup P,\, X\cup Q\}$
is weakly $r$-separated as well.
  \end{theorem}

(Obviously, $P$ and $Q$ are not weakly $r$-separated, and $|P\cup Q|=r+2$
easily implies that any two sets in $\Nscr\cup\{P,Q\}$ except for $P,Q$ are
weakly $r$-separated.) In general, for two weakly $r$-separated collections
$\Wscr$ and $\Wscr'$, if there are $P,Q,X$ as above such that
$\Wscr'=(\Wscr-\{X\cup P\})\cup \{X\cup Q\}$ and $\Wscr=(\Wscr'-\{X\cup
Q\})\cup \{X\cup P\}$, then we say that $\Wscr'$ is obtained from $\Wscr$ by a
\emph{raising} (combinatorial) \emph{flip}, while $\Wscr$ is obtained from
$\Wscr'$ by a \emph{lowering flip}.

Our method of proof of the above theorems (and more) appeals to a geometric
approach and uses some facts on fine zonotopal tilings, or \emph{cubillages},
on a \emph{cyclic zonotope} in a space $\Rset^d$. One of them is that the
maximal by size (strongly) $(d-1)$-separated collections $\Sscr$ in $2^{[n]}$
one-to-one correspond to the cubillages $Q$ in a cyclic zonotope $Z(n,d)$
generated by (a cyclic configuration of) $n$ vectors in $\Rset^d$; one may say
that the set of vertices of $Q$ ``encodes'' $\Sscr$. (When $d=2$, a cubillage
becomes a rhombus tiling on a planar $n$-zonogon, and a bijection between these
tilings and the maximal strongly separated collections in $2^{[n]}$ is
well-known. For $d=3$, a bijection between the corresponding cubillages and
maximal 2-separated sets was originally established in~\cite{gal}. For a
general $d$, the corresponding bijection was recently shown by Galashin and
Postnikov~\cite{GP}.)

Another useful fact, inspired by a result in the classical work due to Manin
and Schechtman~\cite{MS} on higher Bruhat orders, is that any cubillage on
$Z(n,d-1)$ can be lifted as a certain $(d-1)$-dimensional subcomplex, that we
call an \emph{s-membrane}, in some cubillage on $Z(n,d)$. For more details and
other relevant facts, see~\cite{DKK3,zieg}.

We further develop the theory of cubillages by constructing a certain
\emph{fragmentation} $\Qfrag$ of a cubillage $Q$ on $Z(n,d)$, introducing a
class of $(d-1)$-dimensional subcomplexes in $\Qfrag$, called
\emph{w-membranes}, and showing (in Theorem~\ref{tm:membr-wsepar}) that when
$d$ is odd, the vertex set of any w-membrane forms a maximal by size weakly
$(d-2)$-separated collection in $2^{[n]}$. It turns out that the collections of
this sort (over all cubillages on $Z(n,d)$) constitute a poset with a unique
minimal element and a unique maximal element and where neighboring collections
are linked by flips; this is obtained as a consequence of
Theorems~\ref{tm:membr-wsepar} and~\ref{tm:gen_witn}.

In light of this, given an odd $r$ and $n>r$, we can specify three classes
$\bfW_{n,r}$, $\bfW^=_{n,r}$ and $\bfW^\ast_{n,r}$ of weakly $r$-separated
collections $\Wscr$ in $2^{[n]}$, in which $\Wscr$ is maximal by inclusion,
maximal by size, and representable, respectively. Here we call $\Wscr$
\emph{representable} if it can be represented as the vertex set of a w-membrane
in a cubillage on $Z(n,r+2)$ (in particular, $\Wscr$ is maximal by size). Then
$\bfW_{n,r}\supseteq\bfW^=_{n,r}\supseteq\bfW^\ast_{n,r}$.
 \smallskip

This paper is organized as follows. Sect.~\SEC{prelim} contains basic
definitions and reviews some useful facts on cyclic zonotopes and cubillages.
Sect.~\SEC{smembr} recalls the construction of s-membranes in cubillages and
describes their properties needed to us. Here we also introduce the so-called
\emph{bead-thread} relation on vertices of a cubillage, which is used in the
proof of Theorem~\ref{tm:w=s}. Sect.~\SEC{proof1} is devoted to proving
Theorem~\ref{tm:w=s}, and Sect.~\SEC{proof2} proves a sharper version of
Theorem~\ref{tm:gen_witn} (given in Theorem~\ref{tm:PQYstr} and
Corollary~\ref{cor:gen_witPQ} where instead of the whole set $\Nscr(P,Q)$ of
neighbors of $P,Q$ we take into account only those neighbors that are at
distance $\le 2$ from $P$ or $Q$).

Sect.~\SEC{weaksep-cubil} introduces w-membranes in the fragmentation of a
cubillage and proves the above-mentioned results on w-membranes in a cubillage
on $Z(n,d)$ and representable $(d-2)$-separated collection in $2^{[n]}$, and on
the poset of such collections (Theorem~\ref{tm:membr-wsepar} and
Corollary~\ref{cor:posetWast}). Sect.~\SEC{non-pure} demonstrates an example of
``non-pure'' weakly $r$-separated collections $\Wscr$ in $2^{[n]}$, in the
sense that $\Wscr\in\bfW_{n,r}$ but $|\Wscr|<w_{n,r}$, thus showing that the
inclusion $\bfW_{n,r}\supseteq\bfW^=_{n,r}$ can be strict (in contrast to the
well-known purity result for $r=1$, saying that $\bfW_{n,1}=\bfW^=_{n,1}$).
Also we raise two conjectures on weakly $r$-separated set-systems (in
Sects.~\SEC{weaksep-cubil} and~\SEC{non-pure}).

Proofs of two propositions stated in Sect.~\SEC{weaksep-cubil} are given in
Appendix~A. The paper finishes with Appendix~B where we discuss a reasonable
analog of the weak $r$-separation when $r$ is even, outline some constructions
and results on this way and raise two more conjectures.


\section{Preliminaries}  \label{sec:prelim}

This section contains additional definitions, notation and conventions that
will be needed later on. Also we review some known properties of cubillages.
 \smallskip

\noindent $\bullet$ ~Let $n,d$ be positive integers with $n\ge d>1$. By a
\emph{cyclic configuration} of size $n$ in $\Rset^d$ we mean an ordered set
$\Xi$ of $n$ vectors $\xi_i=(\xi_i(1),\ldots,\xi_i(d))\in\Rset^d$,
$i=1,\ldots,n$, satisfying:
  \begin{numitem1} \label{eq:cyc_conf}
\begin{itemize}
\item[(a)] $\xi_i(1)=1$ for each $i$, and
\item[(b)] for the $d\times n$ matrix $A$ formed by $\xi_1,\ldots,\xi_n$ as columns
(in this order), any flag minor of $A$ is positive.
  \end{itemize}
  \end{numitem1}

A typical (and commonly used) sample of such configurations $\Xi$ is generated
by the Veronese curve; namely, take reals $t_1<t_2<\cdots<t_n$ and assign
$\xi_i:=\xi(t_i)$, where $\xi(t)=(1,t,t^2,\ldots,t^{d-1})$.

The \emph{zonotope} $Z=Z(\Xi)$ generated by $\Xi$ is the Minkowski sum of line
segments $[0,\xi_i]$, $i=1,\ldots,n$. A \emph{fine zonotopal tiling} is a
subdivision $Q$ of $Z$ into $d$-dimensional parallelotopes such that: any two
intersecting ones share a common face, and each face of the boundary of $Z$ is
contained in some of these parallelotopes. For brevity, we liberally refer to
these parallelotopes as \emph{cubes}, and to $Q$ as a \emph{cubillage}.

\noindent $\bullet$ ~When $n,d$ are fixed, the choice of one or another cyclic
configuration $\Xi$ (subject to~\refeq{cyc_conf}) does not matter in essence,
and for this reason, we unify notation $Z(n,d)$ for $Z(\Xi)$, referring to it
as the \emph{cyclic zonotope} for $(n,d)$.
 \smallskip

\noindent $\bullet$ ~Let $\pi$ denote the projection $\Rset^d\to\Rset^{d-1}$
given by $(x_1,\ldots,x_d)\mapsto (x_1,\ldots,x_{d-1})$. Due
to~\refeq{cyc_conf}, the vectors $\pi(\xi_1),\ldots,\pi(\xi_n)$ form a cyclic
configuration as well, and we may say that $\pi$ projects $Z(n,d)$ to the
zonotope $Z(n,d-1)$.
 \smallskip

\noindent $\bullet$ ~Each subset $X\subseteq [n]$ naturally corresponds to the
point $\sum_{i\in X}\xi_i$ in $Z(n,d)$, and the cardinality $|X|$ is called the
\emph{height}, or \emph{level} of this subset/point. (W.l.o.g., we usually
assume that all combinations of vectors $\xi_i$ with coefficients 0,1 are
different.)
 \smallskip

\noindent $\bullet$ ~Depending on the context, we may think of a cubillage $Q$
on $Z(n,d)$ in two ways: either as a set of $d$-dimensional cubes (and write
$C\in Q$ for a cube $C$ in $Q$) or as the corresponding polyhedral complex. The
0- and 1-dimensional cells (or faces) of $Q$ are called \emph{vertices} and
\emph{edges}, respectively. A simple fact is that, by the subset-to-point
correspondence, each vertex is identified with a subset of $[n]$. In turn, each
edge $e$ is a parallel translation of some segment $[0,\xi_i]$; we say that $e$
has \emph{color} $i$, or is an $i$-\emph{edge}. When needed, $e$ is regarded as
a directed edge (according to the direction of $\xi_i$).
 \smallskip

\noindent $\bullet$ ~Let $V(Q)$ denote the set of vertices of a cubillage $Q$.
Galashin and Postnikov~\cite{GP} showed the following important correspondence
between cubillages and separated set-systems:
  \begin{numitem1} \label{eq:cub-separ}
for any cubillage $Q$ on $Z(n,d)$, the set $V(Q)$ of its vertices (regarded as
subsets of $[n]$) constitutes a maximal by size $(d-1)$-separated collection in
$2^{[n]}$; conversely, for any maximal by size $(d-1)$-separated collection
$\Sscr\subseteq 2^{[n]}$, there exists a cubillage $Q$ on $Z(n,d)$ with
$V(Q)=\Sscr$.
  \end{numitem1}

\noindent $\bullet$ ~When a cell (face) $C$ of $Q$ has the lowest point
$X\subseteq[n]$ and when $T\subseteq[n]$ is the set of colors of edges in $C$,
we say that $C$ has the \emph{root} $X$ and \emph{type} $T$, and may write
$C=(X\,|\,T)$. One easily shows that $X\cap T=\emptyset$. Another well-known
fact is that for any cubillage $Q$, the types of all ($d$-dimensional) cubes in
it are different and form the set $\binom{[n]}{d}$ of $d$-element subsets of
$[n]$ (so $Q$ has exactly $\binom{n}{d}$ cubes).
 \smallskip

\noindent $\bullet$ ~For a closed subset $U$ of points in $Z=Z(n,d)$, let
$U^{\rm fr}$ ($U^{\rm rear}$) be the part of $U$ ``seen'' in the direction of
the last, $d$-th, coordinate vector $e_d$ (resp. $-e_d$), i.e., the set formed
by the points $x\in\pi^{-1}(x')\cap U$ with $x_d$ minimum (resp. maximum) for
all $x'\in\pi(U)$. It is called the \emph{front} (resp. \emph{rear})
\emph{side} of $U$.

In particular, $\Zfr$ and $\Zrear$ denote the front and rear sides,
respectively, of (the boundary of) the zonotope $Z$. We call $\Zfr\cap\Zrear$
the \emph{rim} of $Z$ and denote it as $\Zrim$.
 \smallskip

\noindent $\bullet$ ~When a set $X\subseteq[n]$ is the union of $k$ intervals
and $k$ is as small as possible, we say that $X$ is a $k$-\emph{interval}. Note
that for such an $X$, its complementary set $[n]-X$ is a $k'$-interval with
$k'\in\{k-1,k,k+1\}$. In the next section we will use the following known
characterization of the sets of vertices in the front and rear sides of a
zonotope of an odd dimension (cf., e.g.,~\cite{DKK3}).
  \begin{numitem1} \label{eq:Zfr-Zrear}
Let $d$ be odd. Then for $Z=Z(n,d)$,
 \begin{itemize}
\item[(i)] ~$V(\Zfr)$ is formed by all $k$-intervals of $[n]$ with $k\le (d-1)/2$;
\item[(ii)] ~$V(\Zrear)$ is formed by the subsets of $[n]$ complementary to those in $V(\Zfr)$;
specifically, it consists of all $k$-intervals with $k<(d-1)/2$, all
$(d-1)/2$-intervals containing at least one of the elements 1 and $n$ and all
$(d+1)/2$-intervals containing both 1 and $n$.
\end{itemize}
    \end{numitem1}
This implies that $V(\Zrim)$ consists of the $k$-intervals with $k<(d-1)/2$ and the
$(d-1)/2$-intervals containing at least one of 1 and $n$;  the set of
\emph{inner} vertices in $\Zfr$, i.e., $V(\Zfr)-V(\Zrim)$ consists of the
$(d-1)/2$-intervals containing none of 1 and $n$, whereas $V(\Zrear)-V(\Zrim)$
consists of the $(d+1)/2$-intervals containing both 1 and $n$.
 \smallskip

\noindent $\bullet$ ~ Consider a cube $C=(X\,|\,T)$ and let
$T=(p(1)<p(2)<\cdots<p(d))$. This cube has $2d$ facets
$F_1,\ldots,F_d,G_1,\ldots,G_d$, where
  \begin{numitem1} \label{eq:cube_facets}
$F_i=F_i(C)$ is viewed as $(X\,|\,T-p(i))$, and $G_j=G_j(C)$ as
$(Xp(j)\,|\,T-p(j))$.
  \end{numitem1}
(For a set $A$ and an element $a\in A$, we abbreviate $A-\{a\}$ to $A-a$.)


\section{S-membranes and bead-threads} \label{sec:smembr}

In this section we recall the definition of s-membranes, associate with a
cubillage a certain path structure, and review some basic properties.
\medskip

\noindent\textbf{Definition.} ~Let $Q$ be a cubillage on $Z(n,d)$. An
\emph{s-membrane} in $Q$ is a subcomplex $M$ of $Q$ such that $M$ (regarded as
a subset of $\Rset^d$) is \emph{bijectively} projected by $\pi$ to $Z(n,d-1)$.
  \smallskip

Then each facet of $Q$ occurring in $M$ is projected to a cube of dimension $d-1$ in $Z(n,d-1)$
and these cubes constitute a cubillage on $Z(n,d-1)$, denoted as $\pi(M)$. In
view of~\refeq{cub-separ} and~\refeq{snr} (applied to $\pi(Q)$), we obtain that
  \begin{numitem1} \label{eq:size_membr}
all s-membranes $M$ in a cubillage $Q$ on $Z(n,d)$ have the same number of
vertices, which is equal to $s_{n,d-2}$, and the vertex set of $M$ (regarded as
a collection in $2^{[n]}$) is $(d-2)$-separated.
  \end{numitem1}

Two s-membranes are of an especial interest. These are the front side $\Zfr$
and the rear side $\Zrear$ of $Z=Z(n,d)$ (in these cases the choice of a
cubillage on $Z$ is not important.) Following terminology in~\cite{DKK2,DKK3},
their projections $\pi(\Zfr)$ and $\pi(\Zrear)$ are called the \emph{standard}
and \emph{anti-standard} cubillages on $Z(n,d-1)$, respectively.

Next we distinguish certain vertices in cubes. When $n=d$, the zonotope turns
into the cube $C=(\emptyset|[d])$, and there holds:
 \begin{numitem1} \label{eq:unique_inner}
the front side $\Cfr$ (rear side $\Crear$) of $C=(\emptyset|[d])$ has a unique
inner vertex, namely, $t_C:=\{i\in[n]\colon d-i$ odd$\}$ (resp.
$h_C:=\{i\in[n]\colon d-i$ even$\}$.
  \end{numitem1}

(When $d$ is odd,~\refeq{unique_inner} can be obtained from~\refeq{Zfr-Zrear}. A
direct proof of~\refeq{unique_inner} for an arbitrary $d$ is as follows (a
sketch). The facets of $C$ are $F_i=F_i(C):=(\emptyset| [d]-i)$ and
$G_i=G_i(C):=(i|[d]-i)$, $i=1,\ldots,d$ (cf.~\refeq{cube_facets}). A facet $F_i$
is contained in $\Cfr$ ($\Crear$) if, when looking at the direction $e_d$, $C$
lies ``behind'' (resp. ``before'') the hyperplane containing $F_i$, or,
equivalently, $\det(A_i)>0$ (resp. $\det(A_i)<0$), cf.~\refeq{cyc_conf}(b),
where $A_i$ is the matrix with the columns
$\xi_1,\ldots,\xi_{i-1},\xi_{i+1},\ldots \xi_d,\xi_i$ (in this order). It
follows that $F_i\subset \Cfr$ if and only if $d-i$ is even. By ``central
symmetry'', $G_i\subset \Cfr$ if and only if $d-i$ is odd.

Now consider a vertex $X\subseteq [d]$ of $C$. If $X$ (resp. $[d]-X$) has
consecutive elements $i-1$ and $i$, then $X\in G_{i-1}$ and simultaneously $X\in G_i$
(resp. $X\in F_{i-1}$ and $X\in F_i$). This implies that $X$ is in both $\Cfr$
and $\Crear$, i.e., $X\in \Crim$. The remaining vertices of $C$ are just $t_C$
and $h_C$ as in~\refeq{unique_inner}; one can see that the former (latter) is
contained in all facets $F_j$ and $G_i$ with $d-j$ even and $d-i$ odd (resp.
$d-j$ odd and $d-i$ even). So $t_C$ lies in $\Cfr$, and $h_C$ in $\Crear$;
moreover, both are not in $\Crim$ (since $C$ is full-dimensional).)
  \medskip

When $n$ is arbitrary and $Q$ is a cubillage on $Z=Z(n,d)$, we distinguish
vertices $t_C$ and $h_C$ of a cube $C\in Q$ in a similar way; namely
(cf.~\refeq{unique_inner}),
 \begin{numitem1} \label{eq:gencube_inner}
if $C=(X\,|\,T)$ and $T=(p_1<\ldots <p_d))$, then $t_C=X\cup\{p_i\colon d-i$
odd$\}$ and $h_C=X\cup\{p_i\colon d-i$ even$\}$.
  \end{numitem1}

Also for each vertex $v$ of $Q$, unless $v$ is in $\Zrear$, there is a unique
cube $C\in Q$ such that $t_C=v$, and symmetrically, unless $v$ is in $\Zfr$,
there is a unique cube $C\in Q$ such that $h_C=v$ (to see this, consider the
line going through $v$ and parallel to $e_d$).

Therefore, by drawing for each cube $C\in Q$, the edge-arrow from $t_C$ to
$h_C$, we obtain a directed graph whose connectivity components are directed
paths beginning at $\Zfr-\Zrim$ and ending at $\Zrear-\Zrim$. We call these
paths \emph{bead-threads} in $Q$. It is convenient to add to this graph the
elements of $V(\Zrim)$ as isolated vertices, forming \emph{degenerate}
bead-threads, each going from a vertex to itself. Let $B_Q$ be the resulting
directed graph. Then
  \begin{numitem1} \label{eq:BQ}
~$B_Q$ contains  all vertices of $Q$, and each component of $B_Q$ is a
bead-thread going from $\Zfr$ to $\Zrear$.
  \end{numitem1}

Note that the heights $|X|$ of vertices $X$ along a bead-thread are monotone
increasing when $d$ is odd, and constant when $d$ is even.


\section{Proof of Theorem~\ref{tm:w=s}}  \label{sec:proof1}

Let $r$ be odd and $n>r$. We have to show that
  \begin{numitem1} \label{eq:sizeW}
~if $\Wscr$ is a weakly $r$-separated collection in $2^{[n]}$, then $|\Wscr|\le
\binom{n}{\le r+1}$.
  \end{numitem1}

This is valid when $r=1$ (cf.~\refeq{max_size}) and is trivial when $n=r+1$. So
one may assume that $3\le r\le n-2$. We prove~\refeq{sizeW} by induction,
assuming that the corresponding inequality holds for $\Wscr',n',r'$ when $n'\le
n$, $r'\le r$, and $(n',r')\ne (n,r)$.

Define the following subcollections in $\Wscr$:
  \begin{eqnarray}
  \Wscr^-&:=& \{A\subseteq [n-1]\,\colon\,
  \{A,An\}\cap\Wscr\ne\emptyset\},\quad \mbox{and} \nonumber \\
  \Tscr&:=& \{A\subseteq [n-1]\,\colon\, \{A,An\}\subseteq \Wscr\}, \nonumber
  \end{eqnarray}
referring to elements $A,An$ in $\Wscr$ as \emph{twins}. Observe that
  \begin{numitem1} \label{eq:ABW-}
~any $A,B\in \Wscr^-$ are weakly $r$-separated.
  \end{numitem1}

Indeed, this is trivial when $A,B\in\Wscr$ or $An,Bn\in\Wscr$. Assume that
$A\in\Wscr$ and $B':=Bn\in\Wscr$, and that $A,B'$ are $(r+2)$-interlaced (for if
$A,B'$ are $r'$-interlaced with $r'\le r+1$, then so is for $A,B$, and we are
done). Since $\max(B'-A)=n>\max(A-B')$ and $r+2$ is odd, $B'$ surrounds $A$.
Therefore, $\min(B'-A)<\min(A-B')$ and $|B'|\le |A|$. Then $|B|<|A|$ and
$\min(B-A)=\min(B'-A)<\min(A-B)$, implying that $A,B$ are weakly $r$-separated,
as required.

By induction, $|\Wscr^-|\le\binom{n-1}{\le r+1}$. Also one can observe that
$|\Wscr|=|\Wscr^-|+|\Tscr|$. Therefore, using the identity
$\binom{n}{j}=\binom{n-1}{j} +\binom{n-1}{j-1}$ for any $j\le n-1$, in order to
obtain the inequality in~\refeq{sizeW}, it suffices to show that
  \begin{equation} \label{eq:Tscr}
  |\Tscr|\le\small{\binom{n-1}{\le r}}.
  \end{equation}

For $i=0,1,\ldots n-1$, define $\Tscr^i:=\{A\in\Tscr\,\colon\, |A|=i\}$. We
will rely on two claims.
 \medskip

\noindent\textbf{Claim 1} ~\emph{For each $i$, the collection $\Tscr^i$ is
$(r-1)$-separated; moreover, $\Tscr^i$ is weakly $(r-2)$-separated}.
 \medskip

\noindent\textbf{Proof} ~Let $A,B\in\Tscr^i$. Take the interval cortege
$(I_1,\ldots,I_{r'})$ for $A,B$, and let for definiteness $I_{r'}\subseteq
A-B$. Then $(I_1,\ldots,I_{r'},I_{r'+1}:=\{n\})$ is the interval cortege for
$A$ and $B':=Bn$. Since $|A|<|B'|$ and $\max (A-B')<\max(B'-A)=n$, the fact
that $A,B'$ are weakly $r$-separated implies that $r'+1$ is strictly less than
$r+2$. Then $r'\le r$, which means that $A,B$ are $(r-1)$-separated. Since
$|A|=|B|$ and $r$ is odd, we also can conclude that $A,B$ are weakly
$(r-2)$-separated. \hfill\qed
 \medskip

Now consider the zonotope $Z=Z(n-1,r)$. For $j=0,1,\ldots,n-1$, define
$\Sscr^j$ ($\Ascr^j$) to be the set of vertices $X$ of $\Zfr$ (resp. $\Zrear$)
with $|X|=j$. We extend each collection $\Tscr^i$ to $\Dscr^i$, defined as
  \begin{equation} \label{eq:Dscr_i}
\Dscr^i:=\Tscr^i\cup (\Sscr^{i+1}\cup\ldots \cup \Sscr^{n-1}) \cup
      (\Ascr^0\cup\Ascr^1\cup\ldots \cup\Ascr^{i-1}).
  \end{equation}

\noindent\textbf{Claim 2} ~\emph{$\Dscr^i$ is weakly $(r-2)$-separated.}
\medskip

\noindent\textbf{Proof} ~The vertex sets of $\Zfr$ and $\pi(\Zfr)$ are
essentially the same (regarding a vertex as a subset of $[n-1]$), and similarly
for $\Zrear$ and $\pi(\Zrear)$. Since $\pi(\Zfr)$ and $\pi(\Zrear)$ are
cubillages on $Z(n-1,r-1)$ (the so-called ``standard'' and ``anti-standard''
ones), ~\refeq{cub-separ} implies that both collections
$V(\Zfr)=\Sscr^0\cup\ldots\cup \Sscr^{n-1}$ and
$V(\Zrear)=\Ascr^0\cup\ldots\cup \Ascr^{n-1}$ are $(r-2)$-separated, and
therefore, they are weakly $(r-2)$-separated as well.

Next, by~\refeq{Zfr-Zrear}(i), each vertex $X$ of $\Zfr$ is a $k$-interval with
$k\le(r-1)/2$. Such an $X$ and any subset $Y\subseteq[n-1]$ are $k'$-interlaced
with $k'\le 2k+1$. Then $k'\le r$ and this holds with equality when $X$ and $Y$
are $r$-interlaced and $Y$ surrounds $X$. It follows that $X$ is weakly
$(r-2)$-separated from any $Y\subseteq[n-1]$ with $|Y|\le|X|$ (in particular,
if $X\in\Sscr^j$ and $j\ge i$, then $X$ is weakly $(r-2)$-separated from each
member of $\Tscr^i\cup\Ascr^0\cup\ldots\cup \Ascr^{i-1}$).

Symmetrically, by~\refeq{Zfr-Zrear}(ii), each vertex $X$ of $\Zrear$ is the
complement (to $[n-1]$) of a $k$-interval with $k\le(r-1)/2$. We can conclude
that such an $X$ is weakly $(r-2)$-separated from any $Y\subseteq[n-1]$ with
$|Y|\ge|X|$.

Now the result is provided by the inequalities $|X|>|A|>|X'|$ for any $X\in
\Sscr^{i+1}\cup\ldots\cup\Sscr^{n-1}$, \;$A\in\Tscr^i$, and
$X'\in\Ascr^0\cup\ldots\cup \Ascr^{i-1}$. \hfill\qed
 \medskip

By induction, $|\Dscr^i|\le\binom{n-1}{\le r-1}$. Then, using~\refeq{snr}
and~\refeq{size_membr} (relative to $n-1$ and $r-2$), we have
  \begin{equation} \label{eq:Di_VZfr}
|\Dscr^i|\le\small{\binom{n-1}{\le r-1}}=s_{n-1,r-2}=|V(\Zfr)|.
  \end{equation}

Let $\Sscr':=\Sscr^0\cup\Sscr^1\cup\ldots\cup \Sscr^i$ and
$\Ascr':=\Ascr^0\cup\Ascr^1\cup\ldots\cup\Ascr^{i-1}$. Since $\Sscr^{i+1}\cup
\ldots\cup \Sscr^{n-1}=V(\Zfr)-\Sscr'$, we obtain from~\refeq{Dscr_i}
and~\refeq{Di_VZfr} that
  \begin{equation} \label{eq:T-D-S-A}
  |\Tscr^i|=|\Dscr^i|-(|V(\Zfr)-\Sscr'|)-|\Ascr'|\le |\Sscr'|-|\Ascr'|.
  \end{equation}

We now finish the proof by using a bead-thread techniques (as in
Sect.~\SEC{smembr}). Fix an arbitrary cubillage $Q$ in $Z=Z(n-1,r)$. Let
$\Rscr^i$ be the set of vertices $X$ of $Q$ with $|X|=i$, and let $\Bscr$ be
the set of paths (bead-threads) in the graph $B_Q$ beginning at $\Zfr$ and
ending at $\Zrear$. Since $r$ is odd, each edge $(X,Y)$ of $B_Q$ is
``ascending'' (satisfies $|Y|>|X|$). This implies that each path $P\in\Pscr$
beginning at $\Sscr'$ must meet either $\Rscr^i$ or $\Ascr'$, and conversely,
each path meeting $\Rscr^i\cup\Ascr'$ begins at $\Sscr'$. This
and~\refeq{T-D-S-A} imply
     $$
  |\Tscr^i|\le|\Rscr^i|.
  $$

Summing up these inequalities for $i=0,1,\ldots,n-1$, we have
  $$
 |\Tscr|=\sum\nolimits_i |\Tscr^i|\le\sum\nolimits_i|\Rscr^i| =|V_Q|= s_{n-1,r-1}
  =\small{\binom{n-1}{\le r}},
  $$
yielding~\refeq{Tscr} and completing the proof of Theorem~\ref{tm:w=s}.
\hfill\qed\qed


\section{Proof of Theorem~\ref{tm:gen_witn}}  \label{sec:proof2}

Let $r,\,r',\,P=\{p_1,\ldots,p_{r'}\},\, Q=\{q_0,\ldots,q_{r'}\}$ and $X$ be as
in the hypotheses of Theorem~\ref{tm:gen_witn} (where $r$ is odd and
$r'=(r+1)/2$). We will use the following notation and terminology.

Let $A,B\subset[n]$. The interval cortege for $A,B$ is denoted by $\Iscr(A,B)$,
and when it is not confusing, we refer to the intervals in it concerning $A-B$
($B-A$) as $A$-\emph{bricks} (resp. $B$-\emph{bricks}). When $A\cap
B=\emptyset$, we may abbreviate $A\cup B$ as $AB$. When $A,B$ are not weakly
$r$-separated, we say that the pair $\{A,B\}$ is \emph{bad}.

Note that for $P,Q,X$ as above and for the set $\Nscr(P,Q)$ of neighbors of
$P,Q$ (defined in~\refeq{neighb}), there hold: $PX$ and $QX$ are
$(r+2)$-interlaced; $XQ$ surrounds $XP$; ~$|XQ|>|XP|$; and $\{XP,XQ\}$ is the
unique bad pair in the collection $\{XS\colon S\in \{P,Q\}\cup\Nscr(P,Q)\}$.

We are going to obtain a sharper version of Theorem~\ref{tm:gen_witn}. In
particular, it deals with only $O(r^2)$ (rather than exponentially many)
neighbors of $(P,Q)$.

  \begin{theorem} \label{tm:PQYstr}
Let $r,n,P,Q,X$ be as above. Define
  \begin{eqnarray}
  \Nscr^\uparrow(P,Q) &:=& \{Pq\colon q\in Q\}\cup \{(P-p)q\colon
  p\in P,\, q\in Q\}; \quad\mbox{and} \label{eq:NupP}\\
   \Nscr^\downarrow(P,Q) &:=& \{Q-q\colon q\in Q\}\cup \{(Q-q)p\colon
  p\in P,\, q\in Q\}. \label{eq:NdownQ}
  \end{eqnarray}
Let $Y\subset[n]$ be different from $XP$ and $XQ$. Then:

{\rm(i)} if $\{Y,XP\}$ is bad, then there exists $S\in\ \Nscr^\uparrow(P,Q)$
such that $\{Y,XS\}$ is bad;

{\rm(ii)} if $\{Y,XQ\}$ is bad, then there exists $S\in\ \Nscr^\downarrow(P,Q)$
such that $\{Y,XS\}$ is bad.
  \end{theorem}

(To obtain Theorem~\ref{tm:gen_witn}, consider $\Wscr,U$ as in that
theorem and let $U'$ be the member of $\{XP,XQ\}$ different from $U$. Suppose
that $\{Y,U'\}$ is bad for some $Y\in\Wscr-\{U\}$. By Theorem~\ref{tm:PQYstr}
applied to $Y,U'$, there exists $S\in\Nscr(P,Q)$ such that $\{Y,XS\}$ is bad.
But $\Wscr$ is weakly $r$-separated and contains both $Y$ and $XS$.)
 \medskip

\begin{proof}
W.l.o.g., one may assume that $Y\cap X=\emptyset$. We first prove assertion~(i)
(obtaining (ii) as a consequence, as we explain in the end of the proof). We
will abbreviate the neighbor set $\Nscr^\uparrow(P,Q)$ as $\Nscrup$, and the
interval cortege $\Iscr(Y,XP)$ as $\Iscr$.

Suppose, for a contradiction, that no pair $\{Y,XS\}$ with $S\in\Nscrup$ is
bad. This will impose restrictions on $Y$ and eventually will lead us to the
conclusion that $Y$ is impossible.

The core of the proof consists in the next lemma. Here we refer to an element
$p\in P$ ($q\in Q$) as \emph{refined} if it forms the single-element $XP$-brick
$\{p\}$ (resp. the single-element $Y$-brick $\{q\}$) in $\Iscr$.

 \begin{lemma} \label{lm:refinedPQ}
Let $Y\subset[n]$ be different from $XP,XQ$ and suppose that for each
$S\in\Nscrup(P,Q)$, the sets $Y$ and $XS$ are weakly $r$-separated (while
$\{Y,XP\}$ is bad). Then at least one of the following holds:

{\rm($\ast$)} ~all elements of $P$ are refined;

{\rm($\ast\ast$)} ~all elements of $Q$ are refined.
 \end{lemma}

This lemma will be proved later, and now assuming its validity, we finish the
proof of the theorem as follows. Note that $Y\cap P\ne \emptyset$ is possible
(whereas $Y\cap X=\emptyset$, as is assumed above).
 \smallskip

Let $a$ and $b$ denote the numbers of $Y$- and $XP$-bricks in $\Iscr$,
respectively. Then $a+b=|\Iscr|\ge r+2=2r'+1$ and $|a-b|\le 1$. We assume that
the intervals in $\Iscr$ are viewed as $\ldots
<A_{i-1}<B_i<A_i<B_{i+1}\ldots$\;, where $A_{i'}$ ($B_{i'}$) stands for a
$Y$-brick (resp. $XP$-brick). The first (last) $Y$-brick is denoted  by
$A^m$ (resp. $A^M$), and the first (last) $XP$-brick by $B^m$ (resp. $B^M$). Also for a set $C\subset [n]$ and a singleton $c\in[n]$, we write $c<C$ ($c>C$) if $c<\min(C)$ (resp. $c>\max(C)$).

We first assume that $(\ast\ast)$ from Lemma~\ref{lm:refinedPQ} is valid. Then
$a\ge |Q|=r'+1$. Consider two possible cases for $a$.
 \medskip

\noindent\underline{\emph{Case I}}: $a\ge r'+2$. Then $b\ge r'+1$ and
$|\Iscr|\ge 2r'+3$. If $q_0<B^m$, then, obviously, $A^m=\{q_0\}$. Taking
$S:=Pq_0\in\Nscrup(P,Q)$, we obtain $|\Iscr(Y,XS)|=|\Iscr|-1\ge 2r'+2$ (since
the $Y$-brick $\{q_0\}$ disappears, while the other bricks of $\Iscr$
preserve). Hence $\{Y,XS\}$ is bad. Similarly, if $B^M<q_{r'}$, then
$A^M=\{q_{r'}\}$, and taking $S:=Pq_{r'}$, we again obtain $|\Iscr(Y,XS)|\ge
2r'+2$, whence $\{Y,XS\}$ is bad.

So we may assume that $B^m<q_0$ and $q_{r'}<B^M$. Then $b\ge r'+2$ and
$|\Iscr|=a+b\ge 2r'+4$. Taking $S:=Pq_0$, we obtain $|\Iscr(Y,XS)|=|\Iscr|-2\ge
2r'+2$ (since the $Y$-brick $\{q_0\}$ disappears and the $X..$-bricks preceding
and succeeding $\{q_0\}$ merge). Thus, in all cases, $\{Y,XS\}$ is bad; a
contradiction.
 \medskip

\noindent\underline{\emph{Case II}}: $a=r'+1$. Then $A^m=\{q_0\},\{q_1\},
\ldots, \{q_{r'}\}=A^M$ are exactly the $Y$-bricks of $\Iscr$. If $B^m<q_0$ and
$B^M<q_{r'}$, then $b=a=r'+1$ and $|\Iscr|=2r'+2$. Taking $S:=Pq_{r'}$, we
obtain $|\Iscr(Y,XS)|=|\Iscr|-1= 2r'+1$. Also $XS$ surrounds $Y$ (since $B^M$
becomes the last interval in $\Iscr(Y,XS)$). Hence $|Y-XS|=r'<r'+1\le |XS-Y|$,
implying that $\{Y,XS\}$ is bad.

Similarly, if $q_0<B^m$ and $q_{r'}<B^M$, then $S:=Pq_0$ gives
$|\Iscr(Y,XS)|=|\Iscr|-1= 2r'+1$, and $XS$ surrounds $Y$ as well. And if $B^m<
q_0$ and $q_{r'}<B^M$, then $b=r'+2$, and for $S:=Pq_0$, we obtain
$|\Iscr(Y,XS)|=|\Iscr|-2= 2r'+1$. Again, $XS$ surrounds $Y$, whence $\{Y,XS\}$
is bad.

So it remains to consider the situation when $q_0<B^m$ and $B^M<q_{r'}$. Then $b=r'$ and $Y$ surrounds $XP$.  Since
$\{Y,XP\}$ is bad and $Y-XP=Q$, we have $r'+1=|Y-XP|>|XP-Y|\ge r'$. It
follows that $|XP-Y|=r'$. This implies that each $XP$-brick is a singleton, and that if $Y\cap P=\emptyset$, then the $XP$-bricks of
$\Iscr$ are exactly $\{p_1\},\ldots,\{p_{r'}\}$. But then $X=\emptyset$ and
$Y=Q=XQ$, contradicting the hypotheses of the theorem.

Therefore, $Y$ must contain an element $p_i$ for some $i$. Then $p_i\notin B_i$
(in view of $|B_i|=1$ and $p_i\notin XP-Y$). So one of two situations takes
place: $q_{i-1}<p_i<B_i<q_i$, or $q_{i-1}<B_i<p_i<q_i$. Define
$S:=(P-p_i)q_{i-1}$ in the former case, and $S:=(P-p_i)q_i$ in the latter case.
The transformation $XP\mapsto XS$ replaces the $Y$-brick $\{q_{i-1}\}$ or
$\{q_i\}$ by $\{p_i\}$. We have: $|\Iscr(Y,XS)|=|\Iscr|=2r'+1$, $Y$ surrounds
$XS$, and $|Y|>|XS|$. Hence $\{Y,XS\}$ is bad; a contradiction.
\medskip

Next we assume that $(\ast)$ from Lemma~\ref{lm:refinedPQ} is valid. Then $b\ge
r'$ and each $p_i\in P$ forms the $XP$-brick $\{p_i\}$ in $\Iscr$ (admitting
other $XP$-bricks); in particular, $Y\cap P=\emptyset$. Consider two
possibilities for $b$.
 \medskip

\noindent\underline{\emph{Case III}}: $b=r'$. Then $\{p_1\},\ldots,\{p_{r'}\}$
are exactly the $XP$-bricks of $\Iscr$, $X=\emptyset$, and $a=r'+1$ (in view of
$|\Iscr|\ge 2r'+1$). So $Y$ surrounds $P=XP$, and $|Y|>|P|$. Assuming
that $(\ast\ast)$ from Lemma~\ref{lm:refinedPQ} is violated, there is $i\in
\{0,\ldots,r'\}$ such that $\{q_i\}$ is not a $Y$-brick of $\Iscr$. Then
either~(a) $q_i$ lies in some $Y$-brick $A_j$ with $|A_j|\ge 2$, or~(b) $q_i$
is in none of the intervals of $\Iscr$.

In case~(a), set $S:=Pq_i$. Note that $|Y|\ge r'+2=|P|+2$ (in view of $a=r'+1$
and $|A_j|\ge 2$). If $q_i\in Y$, then the transformation $P\mapsto S$ replaces
$A_j$ by a (possibly smaller but nonempty) $Y$-brick in $\Iscr(Y,S)$, while
preserving the other intervals of $\Iscr$. It follows that
$|\Iscr(Y,S)|=|\Iscr|=2r'+1$, $Y$ surrounds $S$, and $|Y|>|P|+1=|S|$; so
$\{Y,S\}$ is bad. And if $q_i\notin Y$, then, obviously,
$\min(A_j)<q_i<\max(A_j)$. This gives $|\Iscr(Y,S)|=|\Iscr|+2$ (since $P\mapsto
S$ replaces $A_j$ by the $S$-brick $\{q_i\}$ and two $Y$-bricks, one containing
$\min(A_j)$ and the other containing $\max(A_j)$); so $\{Y,S\}$ is bad again.

And in case~(b), obviously, $q_i\notin Y$. Then one of four subcases takes
place: (b1) $p_i<q_i<A_i$; (b2) $A_i<q_i<p_{i+1}$; (b3) $i=0$ and $q_0<A^m$;
and (b4) $i=r'$ and $A^M<q_{r'}$. In subcases~(b3) and~(b4), for $S:=Pq_i$, we
have $|\Iscr(Y,S)|=|\Iscr|+1$ (since $\{q_i\}$ becomes a new brick), whence
$\{Y,S\}$ is bad. In subcases~(b1) and~(b2), for $S':=(P-p_i)q_i$ and
$S':=(P-p_{i+1})q_i$, respectively, the transformation $P\mapsto S'$ replaces
the $P$-brick $\{p_i\}$ or $\{p_{i+1}\}$ by the $S'$-brick $\{q_i\}$, and the
badness of $\{Y,P\}$ implies that of $\{Y,S'\}$.
 \medskip

\noindent\underline{\emph{Case IV}}: $b\ge r'+1$. Assuming, as before, that we
are not in~$(\ast\ast)$ from Lemma~\ref{lm:refinedPQ}, there is $i$ such that
$\{q_i\}$ is not a $Y$-brick of $\Iscr$. Take $S:=Pq_i$. We can observe that in
all possible cases for $q_i$ (exposed in Case~III above), the transformation
$XP\mapsto XS$ leads to the following: $|XS|>|XP|$, each $Y$-brick of $\Iscr$
either preserves or is replaced by a (nonempty) $Y$-brick of
$\Iscr(Y,XS)=:\Iscr'$, and similarly for the $XP$-bricks of $\Iscr$. Then
$|\Iscr'|\ge|\Iscr|\ge 2r'+1$ and, moreover, in case $|\Iscr'|=2r'+1$, the
number of $XS$-bricks ($Y$-bricks) of $\Iscr'$ is the same as the one of
$XP$-bricks (resp. $Y$-bricks) of $\Iscr$, which is equal to $r'+1=b$ (resp.
$r'=a$). So $b>a$, $XP$ surrounds $Y$, and $|XP|>|Y|$ (since $\{Y,XP\}$ is
bad). Now the badness of $\{Y,XS\}$ follows from $|XS|>|XP|$, yielding a
contradiction.
\medskip

Thus, assertion~(i) in the theorem is proven.
\medskip

It remains to show~(ii). We reduce it to the previous case, using the following
observation. For $A\subseteq[n]$, let $\bar A$ denote the complementary set
$[n]-A$. One can see that $\Iscr(A,B)=\Iscr(\bar A,\bar B)$ and that the bricks
for $A-B$ coincide with those for $\bar B-\bar A$. It follows that if $A,B$ are
$(r+2)$-interlaced (where $r$ is odd, as before) and $A$ surrounds $B$, then
$\bar A,\bar B$ are $(r+2)$-interlaced, $\bar B$ surrounds $\bar A$, and
$|A|-|B|=|\bar B|-|\bar A|$. Therefore, if $A,B$ are weakly $r$-separated then
so are $\bar A,\bar B$.

Now for $Y,P,Q,X$ as above and $U:=XQ$, consider $Y':=\bar Y$,
\,$X':=\bar{XPQ}$ ($=[n]-(X\cup P\cup Q)$) and $U':=\bar{XQ}$. Suppose that
$\{Y,U\}$ is bad. Then $\{Y',U'\}$ is bad as well. Note also that $U'=X'P$. By
the theorem applied to $Y',P,Q,X'$ and $U'$, there exists $S'\in\Nscrup(P,Q)$
such that $\{Y',X'S'\}$ is bad. Then $\{Y,\bar{X'S'}\}$ is bad as well. Take
$S:=(P\cup Q)-S'$. One can see that $S\in\Nscrdown(P,Q)$ and $\bar{X'S'}=XS$.
Therefore, $\{Y,XS\}$ is bad, as required.

This completes the proof of Theorem~\ref{tm:PQYstr} (implying
Theorem~\ref{tm:gen_witn}).
 \end{proof}

Note that Theorem~\ref{tm:PQYstr} implies a sharper version of the theorem on
combinatorial flips.

  \begin{corollary} \label{cor:gen_witPQ}
For $r,n,P,Q,X$ as in Theorem~\ref{tm:gen_witn}, if a weakly $r$-separated
collection $\Wscr\subset 2^{[n]}$ contains the set $XP$ ($XQ$) and the sets
$XS$ for all $S\in\ \Nscr^\downarrow(P,Q)$ (resp. $S\in\ \Nscr^\uparrow(P,Q)$),
then the collection obtained from $\Wscr$ by replacing $XP$ by $XQ$ (resp. $XQ$
by $XP$) is weakly $r$-separated as well.
  \end{corollary}

We finish this section with proving the above-mentioned lemma.

\noindent\textbf{Proof of Lemma~\ref{lm:refinedPQ}} ~Suppose that there are
simultaneously $p\in P$ and $q\in Q$ that are not refined. Form $S':=P-p$,
$S'':=Pq$ and $S:=(P-p)q$ (note that $S''$ and $S$ are in
$\Nscr^\uparrow=\Nscr^\uparrow(P,Q)$, whereas $S'$ is not). Let
$\Iscr:=\Iscr(Y,XP)$, $\Iscr':=\Iscr(Y,XS')$ and $\Iscr'':=\Iscr(Y,XS'')$. We
write $A_i$ ($B_i$) for $Y$-bricks (resp. $XP$-bricks) in $\Iscr$ and assume
that they follow in $\Iscr$ in the order $\ldots
<A_{i-1}<B_i<A_i<B_{i+1}\ldots$\;.

For $A,B,A',B'\subseteq[n]$, let us say that the ordered pairs $(A,B)$ and
$(A',B')$ \emph{have the same type} if $|\Iscr(A,B)|=|\Iscr(A',B')|$ and the
first interval of $\Iscr(A,B)$ concerns $A-B$ if and only if the first interval
of $\Iscr(A',B')$ concerns $A'-B'$ (implying that a similar property holds for
the last intervals of $\Iscr(A,B)$ and $\Iscr'(A',B')$).

We examine four possible cases for $p,q$.
 \medskip

\noindent\emph{Case 1}: ~$p$ lies in an interval $C$ of $\Iscr$.

(1a) ~Suppose that $C$ is an $XP$-brick $B_i$. Since $p$ is not refined,
$|B_i|\ge 2$. If $p\notin Y$, then the transformation $XP\mapsto XS'$ replaces
$B_i$ by a (nonempty) $XS'$-brick $B'_i$ in $\Iscr$ with $B'_i\subseteq B_i$, while the other intervals in $\Iscr$ and $\Iscr'$ coincide.
Therefore, $\Iscr'$ and $\Iscr$ have the same type.

And if $p\in Y$, then $\min(B_i)<p<\max(B_i)$, and $XP\mapsto XS'$ replaces
$B_i$ by three bricks, say, $B'<A'<B''$, where $A'$ is the single-element
$Y$-brick $\{p\}$, and $B',B''$ are $XS'$-bricks (with $\min(B')=\min(B_i)$ and
$\max(B'')=\max(B_i)$). Then $|\Iscr'|=|\Iscr|+2$.
\smallskip

(1b) ~Now suppose that $C$ is a $Y$-brick $A_i$. This is possible only if $p\in Y$ and $\min(A_i)<p<\max(A_i)$. Then
$p\in Y-XS'$, and $XP\mapsto XS'$ preserves $A_i$ (as well as the other
intervals of $\Iscr$), whence $\Iscr'=\Iscr$.
\medskip

\noindent\emph{Case 2}: ~$q$ lies in an interval $C$ of $\Iscr$.

(2a) Suppose that $C=B_i$. This is possible only if  $q\notin Y$ and $\min(B_i)<q<\max(B_i)$ (since $q\notin XP$). Then $XP\mapsto XS''$ preserves
$B_i$, yielding $\Iscr''=\Iscr$.
\smallskip

(2b) Now suppose that $C=A_i$. Since $q$ is not refined, $|A_i|\ge 2$. If $q\in
Y$, then $XP\mapsto XS''$ replaces $A_i$ by a (nonempty) $Y$-brick $A'_i$
with $A'_i\subseteq A_i$. Therefore, $\Iscr''$ and $\Iscr$ have the same
type. And if $q\notin Y$, then $XP\mapsto XS''$ replaces $A_i$ by three bricks
$A'<B'<A''$, where $B'$ is the $XS''$-brick $\{q\}$, and $A',A''$ are
$Y$-bricks (with $\min(A')=\min(A_i)$ and $\max(A'')=\max(A_i)$), whence $|\Iscr''|=|\Iscr|+2$.
 \medskip

\noindent\emph{Case 3}: ~$p$ belongs to no interval of $\Iscr$. Then $p\in Y$.

(3a) Suppose that $A_i<p<B_{i+1}$ or $B_i<p<A_i$ for some $i$. Then $p\in
Y-XS'$, and $XP\mapsto XS'$ extends  $A_i$ (making a $Y$-brick with the
beginning or end at $p$). Hence $\Iscr'$ and $\Iscr$ have the same type.
 \smallskip

(3b) Suppose that $p<C$, where $C$ is the first interval of $\Iscr$. If $C$ is
an $XP$-brick, then $XP\mapsto XS'$ produces a new $Y$-brick, namely, $\{p\}$,
and preserves the other intervals of $\Iscr$, whence $|\Iscr'|=|\Iscr|+1$. And
if $C$ is a $Y$-brick, then $XP\mapsto XS'$ extends $C$ (making a $Y$-brick
with the beginning $p$), whence $\Iscr'$ and $\Iscr$ have the same type.
 \smallskip

(3c) Similarly, if $p>D$, where $D$ is the last interval of $\Iscr$, then
either $|\Iscr'|=|\Iscr|+1$, or $\Iscr'$ and $\Iscr$ have the same type (when
$D$ is extended to a $Y$-brick with the end $p$).
  \medskip

\noindent\emph{Case 4}: ~$q$ belongs to no interval of $\Iscr$. Then $q\notin
Y$.

(4a) Suppose that $A_{i-1}<q<B_i$ or $B_i<q<A_i$ for some $i$. Then $XP\mapsto
XS''$ extends $B_i$ (making a $XS''$-brick with the beginning or end at $q$).
Hence $\Iscr'',\Iscr$ have the same type.
 \smallskip

(4b) Suppose that $q<C$, where $C$ is the first interval of $\Iscr$. If $C$ is
an $XP$-brick, then $XP\mapsto XS''$ extends $C$ (making an $XS''$-brick with
the beginning $q$), whence $\Iscr'',\Iscr$ have the same type. And if $C$ is a
$Y$-brick, then $XP\mapsto XS''$ preserves $C$ and produces a new $X..$-brick
$\{q\}$, whence $|\Iscr''|=|\Iscr|+1$.
  \smallskip

(4c) Similarly, if $q>D$, where $D$ is the last interval of $\Iscr$, then
either $\Iscr''$ and $\Iscr$ have the same type, or $|\Iscr''|=|\Iscr|+1$.
\medskip

Now we finish proving the lemma as follows. Observe that in the above cases, $|\Iscr'|\ge |\Iscr|$ is valid throughout, and if this holds with equality, then $\Iscr'$ and $\Iscr$ have the same type. For $\Iscr''$ and $\Iscr$, the behavior is similar. 

If $|\Iscr''|>|\Iscr|$ happens, then $Y$ and $XS''$ form a bad pair (since they
are $|\Iscr''|$-interlaced with $|\Iscr''|>r+2$ and taking into account that
$S''=Pq\in \Nscrup$). This contradicts the hypotheses of the lemma.

Now let $|\Iscr''|=|\Iscr|$ (and $\Iscr'',\Iscr$ have the same type). Then we consider the neighbor $S=(P-p)q\in \Nscrup$ and assert that $\{Y,XS\}$ is bad, thus coming to a contradiction again.

To show this, let $\tilde \Iscr:=\Iscr(Y,XS)$. Suppose that $q\in Y$. Forming $Y^-:=Y-q$, we have $Y^- -XP=Y-XS''$ and $XP-Y^-=XS''-Y$, implying that $\Iscr^-:=\Iscr(Y^-,XP)$ coincides with $\Iscr''$. Hence $\Iscr^-$ and $\Iscr$ have the same type. Moreover, under the correspondence of intervals in these corteges (exposed in~(2b)), each $Y^-$-brick of $\Iscr^-$ is included in the corresponding $Y$-brick of $\Iscr$, and each $XP$-brick of $\Iscr^-$ includes the corresponding $XP$-brick of $\Iscr$. In particular, $p$ is not refined w.r.t. $\Iscr^-$. So we can apply to $X,P,Y^-,p$ reasonings as in Cases~1 and~3 and conclude that under the transformation $XP\mapsto XS'$, the cortege $\Iscr^-$ turns into $\hat \Iscr:=\Iscr(Y^-,XS')$ such that either $|\hat\Iscr|>|\Iscr^-|$, or $\hat\Iscr$ and $\Iscr^-$ have the same type. But $Y=Y^-q$ and $S=S'q$ imply $\hat\Iscr=\tilde\Iscr$. Now the badness of $\{Y,XS\}$ is immediate when $|\hat\Iscr|>|\Iscr^-|$ ($=|\Iscr|$), and follows from the badness of $\{Y,XP\}$ when $|\hat\Iscr|=|\Iscr^-|$ (since $\hat\Iscr$ and $\Iscr$ have the same type and $|Y^-|-|XS'|=|Y|-|XS|=|Y|-|XP|$).

Finally, let $q\notin Y$. Then (in view of $|\Iscr''|=|\Iscr|$) we are in one of the following subcases: (2a) with $\min(B_i)<q<\max(B_i)$ for some $i$; or~(4a) with $A_{i-1}<q<B_i$ or $B_i<q<A_i$ for some $i$; or~(4b) with $q<B^m<A^m$; or~(4c) with $q>B^M>A^M$ (where $A^m$ and $A^M$ (resp. $B^m$ and $B^M$) are the first and last $Y$-bricks (resp. $XP$-bricks) in $\Iscr$, respectively). By explanations above, in all of these situations, $XP\mapsto XS''$ leads to increasing at most one of $X..$-bricks and preserving the other intervals of $\Iscr$. This implies that $p$ is not refined w.r.t. $\Iscr''$, and we can apply to $X,S'',Y,p$ reasonings as in Cases~1 and~3 and conclude that $XS''\mapsto XS$ turns $\Iscr''$ into $\tilde\Iscr$ so that either $|\tilde\Iscr|> |\Iscr''|$ ($=|\Iscr|$), or $\tilde \Iscr$ and $\Iscr''$ have the same type. Then the badness of $\{Y,XS\}$ follows.

This completes the proof of the lemma.
 \hfill\qed


\section{Weakly $r$-separated collections generated by cubillages}  \label{sec:weaksep-cubil}

In Sects.~\SEC{prelim},\,\SEC{smembr} we outlined an interrelation between
(strongly) $\ast$-separated collections on one hand, and cubillages and
s-membranes on other hand (see~\refeq{cub-separ} and~\refeq{size_membr}). This
section is devoted to geometric aspects of the weak $r$-separation, assuming
that $r$ is odd. Being motivated by geometric constructions for maximal weakly
1-separated collections elaborated in~\cite{DKK1,DKK2}, we explain how to
construct maximal by size weakly $r$-separated collections by use of the
so-called \emph{w-membranes}; these are analogs of s-membranes in certain
\emph{fragmentations} of cubillages.

In subsections below we introduce the notions of fragmentation and w-membrane,
demonstrate their properties (extending results from~\cite[Sect.~6]{DKK2}) and
finish with a theorem saying that the vertex set of any $(r+1)$-dimensional
w-membrane gives rise to a maximal by size weakly $r$-separated collection (for
corresponding $n$). Note that in Sects.~\SSEC{fragment}--\SSEC{acycl} the
dimension $d$ of a zonotope/cubillage in question is assumed to be arbitrary
(not necessarily odd).


 \subsection{Fragmentation} \label{ssec:fragment}
~Let $Q$ be a cubillage on $Z(n,d)$. For $\ell=0,1,\ldots,n$, we denote the
``horizontal'' hyperplane at ``height'' $\ell$ in $\Rset^d$ by $H_\ell$, i.e.,
$H_\ell:=\{x=(x_1,\ldots,x_d)\in \Rset^d\,\colon\, x_1=\ell\}$. The
\emph{fragmentation} of $Q$ is meant to be the complex $\Qfrag$ obtained by
cutting $Q$ by $H_1,\ldots,H_{n-1}$.

Such hyperplanes subdivide each cube $C=(X\,|\,T)$ of $Q$ into $d$ pieces
$\Cfrag_1,\ldots,\Cfrag_d$, where $\Cfrag_h$ is the (closed) portion of $C$
between $H_{|X|+h-1}$ and $H_{|X|+h}$. We say that $\Cfrag_h$ is $h$-th
\emph{fragment} of $C$ and, depending on the context, may also think of
$\Qfrag$ as the set of fragments over all cubes. Let $S_h(C)$ denote $h$-th
horizontal \emph{section} $C\cap H_{|X|+h}$ of $C$ (where $0\le h\le d$); this
is the convex hull of the set of vertices
  \begin{equation} \label{eq:hor_sect}
  (X\,|\,\tbinom{T}{h}) \qquad (=\{X\cup A\,\colon\, A\subset T,\; |A|=h\}).
  \end{equation}
(Using terminology of~\cite{MS}, $S_h(C)$ is said to be a \emph{hyper-simplex}.
It turns into a usual simplex when $h=1$ or $d-1$.) Observe that for
$h=1,\ldots,d$,
  \begin{numitem1} \label{eq:fragm_cube}
$h$-th fragment $\Cfrag_h$ of $C$ is the convex hull of the set of vertices
$(X\,|\,\binom{T}{h-1})$ and $(X\,|\,\binom{T}{h}))$; it has two ''horizontal''
facets, namely, $S_{h-1}(C)$ and $S_h(C)$, and $2d$ other facets (conditionally
called ``vertical'' ones), namely, the portions of $F_i(C)$ and $G_i(C)$
between $H_{|X|+h-1}$ and $H_{|X|+h}$ for $i=1,\ldots,d$, denoted as
$F_{h,i}(C)$ and $G_{h,i}(C)$, respectively.
  \end{numitem1}
Here $F_i(C)$ and $G_i(C)$ are the facets of $C=(X|T)$ defined
in~\refeq{cube_facets}, letting $T=(p(1)<p(2)<\cdots<p(d))$. We call
$S_{h-1}(C)$ and $S_h(C)$ the \emph{lower} and \emph{upper} facets of the
fragment $\Cfrag_h$, respectively. Note that $S_0(C)$ and $S_d(C)$ degenerate
to the single points $X$ and $XT$, respectively. The vertical facets
$F_{d,i}(C)$ and $G_{1,i}(C)$ (for all $i$) degenerate as well.

The horizontal facets are ``not fully seen'' under the projection $\pi$. To
visualize all facets of fragments of $\Qfrag$, it is convenient to look at them
as though ``from the front and slightly from below'', i.e., by use of the
projection $\pi^\eps: \Rset^d\to\Rset^{d-1}$ defined by
   \begin{equation} \label{eq:pi_eps}
 x=(x_1,\ldots,x_d)\mapsto (x_1-\eps x_d,x_2,\ldots,x_{d-1})=: \pi^\eps(x)
   \end{equation}
for a sufficiently small $\eps>0$. (Compare $\pi^\eps$ with $\pi$.)
Figure~\ref{fig:pi_eps} illustrates the case $d=3$; here the fragments of a
cube $C=(X\,|\,T)$ with $T=(i<j<k)$ are drawn.

\begin{figure}[htb]
  \vspace{0.2cm}
\begin{center}
\includegraphics[scale=1]{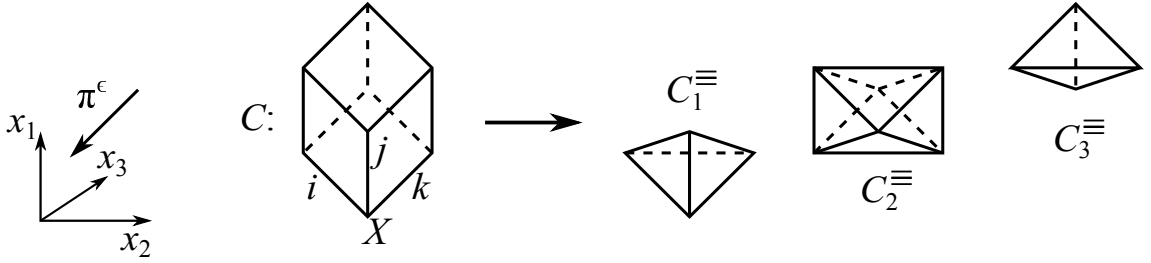}
\end{center}
\vspace{-0.3cm}
 \caption{the fragmentation of cube $C=(X\,|\,T)$}
 \label{fig:pi_eps}
  \end{figure}

Using this projection, we obtain slightly slanted front and rear sides of
objects in $\Qfrag$. More precisely, for a closed set $U$ of points in
$Z=Z(n,d)$, let $U^\epsfr$ ($U^\epsrear$) be the subset of $U$ ``seen'' in the
direction $e_d+\eps e_1$ (resp. $-e_d-\eps e_1$), where $e_i$ is $i$-th
coordinate vector, i.e., formed by the points $x\in(\pi^\eps)^{-1}(x')\cap U$
with $x_d$ minimum (resp. maximum) for all $x'\in\pi^\eps(U)$. We call it the
$\eps$-\emph{front} (resp. $\eps$-\emph{rear}) \emph{side} of $U$.

Obviously, $Z^\epsfr=\Zfr$ and $Z^\epsrear=\Zrear$. Also for a cube $C=(X|T)$
in $Z$, $C^\epsfr=\Cfr$ and $C^\epsrear=\Crear$. As to fragments of $C$, their
$\eps$-front and $\eps$-rear sides are viewed as follows:
\begin{numitem1} \label{eq:sides_fragm}
for $h=1,\ldots,d$, ~$C^\epsfr_h$ is the union of $\Cfr_h$ and the lower facet
$S_{h-1}(C)$ (degenerating to the point $X$ when $h=1$); in its turn,
$C^\epsrear_h$ is formed by the union of $\Crear_h$ and the upper facet
$S_{h}(C)$ (degenerating to the point $XT$ when $h=d$).
  \end{numitem1}

So $C^\epsfr_h\cup C^\epsrear_h$ is just the boundary of $\Cfrag_h$.


 \subsection{W-membranes}  \label{ssec:wmembr}
~Membranes of this sort represent certain $(d-1)$-dimensional subcomplexes of
$\Qfrag$. To introduce them, we consider small deformations of cyclic zonotopes
in $\Rset^{d-1}$ using the projection $\pi^\eps$. More precisely, given a
cyclic configuration $\Xi=(\xi_1,\ldots,\xi_n)$ as in~\refeq{cyc_conf}, define
   $$
   \psi_i:=\pi(\xi_i) \quad\mbox{and}\quad \psi^\eps_i:=\pi^\eps(\xi_i),\quad
   i=1,\ldots, n.
   $$
Then $\Psi=(\psi_1,\ldots,\psi_n)$ obeys~\refeq{cyc_conf} (with $d-1$ instead
of $d$), and when $\eps$ is small enough, $\Psi^\eps=(\psi^\eps_1,\ldots,
\psi^\eps_n)$ obeys the condition~\refeq{cyc_conf}(b), though slightly
violates~\refeq{cyc_conf}(a); yet we keep the term ``cyclic configuration'' for
$\Psi^\eps$ as well. Consider the zonotope in $\Rset^{d-1}$ generated by
$\Psi^\eps$, denoted as $Z^\eps(n,d-1)$ (when it is not confusing).
  \Xcomment{
Assuming, w.l.o.g., that all $\{0,1\}$-combinations of vectors $\psi^\eps_i$
are different, the subsets $X$ of $[n]$ are bijective to the points $v$ of the
form $\sum_{i\in X} \psi^\eps_i$ in $Z^\eps(n,d-1)$; we identify such $v$ and
$X$. (When $\eps$ is small, this gives a natural isomorphism between the
cubillages in $Z^\eps(n,d-1)$ and $Z(n,d-1)$.)
   }
 \smallskip

\noindent \textbf{Definition.} A \emph{w-membrane} of a cubillage $Q$ on
$Z(n,d)$ is a subcomplex $M$ of the fragmentation $\Qfrag$ such that $M$
(regarded as a subset of $\Rset^d$) is bijectively projected by $\pi^\eps$ to
$Z^\eps(n,d-1)$.
  \smallskip

A w-membrane $M$ has facets (of dimension $d-1$) of two sorts, called
\emph{H-tiles} and \emph{V-tiles}. Each H-tile is a horizontal facet of some
fragment (viz. the section $S_h(C)$ of a cube $C$ in $Q$ at height
$h\in[d-1]$). And V-tiles are vertical facets of some fragments $\Cfrag_h$
(see~\refeq{fragm_cube}).


 \subsection{Acyclicity and the lattice structure of w-membranes}  \label{ssec:acycl}
~Let $\bfC(n,d)$ denote the set of all cubes in $Z(n,d)$ (occurring in all
cubillages there). For $C,C'\in\bfC(n,d)$, we say that $C$ \emph{immediately
precedes} $C'$ if $\Crear$ and $(C')^{\rm fr}$ have a common facet. As a far
generalization of the known acyclicity property for cubes in a cubillage, one
can show the following
  \begin{prop} \label{pr:acycl_all}
The directed graph $\Gamma_{n,d}$ whose vertices are the cubes in $\bfC(n,d)$
and whose edges are the pairs $(C,C')$ of cubes such that $C$ immediately
precedes $C'$ is acyclic.
  \end{prop}

(As a consequence, the transitive closure of this ``immediately preceding''
relation forms a partial order on $\bfC(n,d)$.) This proposition enables us to
construct a partial order on the set of fragments for a cubillage $Q$, which in
turn is used to show that the set of w-membranes in $\Qfrag$ forms a
distributive lattice.

More precisely, given a cubillage $Q$ on $Z(n,d)$, consider fragments
$\Delta=\Cfrag_i$ and $\Delta'=(C')^\equiv_j$ of $\Qfrag$. Let us say that
$\Delta$ \emph{immediately precedes} $\Delta'$ if the $\eps$-rear side of
$\Delta$ and the $\eps$-front side of $\Delta'$ share a facet. In other words,
either $C\ne C'$ and $\Delta^{\rm rear}\cap(\Delta')^{\rm fr}$ is a V-tile, or
$C=C'$ and $j=i+1$. The following is important for us.
  \begin{prop} \label{pr:acyc_Qfrag}
The directed graph $\Gamma_{\Qfrag}$ whose vertices are the fragments in
$\Qfrag$ and whose edges are the pairs $(\Delta,\Delta')$ of fragments such
that $\Delta$ immediately precedes $\Delta'$ is acyclic.
  \end{prop}

Proofs of Propositions~\ref{pr:acycl_all} and~\ref{pr:acyc_Qfrag} will be given
in Appendix~A.

From Proposition~\ref{pr:acyc_Qfrag} it follows that the transitive closure of
the above relation on the fragments of $\Qfrag$ forms a partial order; denote
it as $(\Qfrag,\prec)$.

Next we associate with a w-membrane $M$ of $Q$ the set $\Qfrag(M)$ of fragments
in $\Qfrag$ lying in the region of $Z(n,d)$ between $\Zfr$ and $M$. The
constructions of $\pi^\eps$ and $M$ lead to the following property: for
fragments $\Delta,\Delta'$ of $\Qfrag$, if $\Delta$ immediately precedes
$\Delta'$ and if $\Delta'\in \Qfrag(M)$, then $\Delta\in\Qfrag(M)$ as well.
This implies a similar property for fragments $\Delta,\Delta'$ with
$\Delta\prec\Delta'$. So $\Qfrag(M)$ is an ideal of $(\Qfrag,\prec)$. One can
check that a converse property is also true: any ideal of $(\Qfrag,\prec)$ is
expressed as $\Qfrag(M)$ for some w-membrane $M$ of $Q$. Therefore,
   \begin{numitem1} \label{eq:w-lattice}
the set $\Mscrw(Q)$ of w-membranes of a cubillage $Q$ on $Z=Z(n,d)$ is a
distributive lattice in which for $M,M'\in\Mscrw(Q)$, the w-membranes $M\wedge
M'$ and $M\vee M'$ satisfy $\Qfrag(M\wedge M')=\Qfrag(M)\cap \Qfrag(M')$ and
$\Qfrag(M\vee M')=\Qfrag(M)\cup \Qfrag(M')$; the minimal and maximal elements
of this lattice are $\Zfr$ and $\Zrear$, respectively.
  \end{numitem1}

Suppose that $M\in\Mscrw(Q)$ is different from $\Zfr$; then
$\Qfrag(M)\ne\emptyset$. Take a maximal (relative to the order $\prec$ in
$\Qfrag$) fragment $\Delta$ in $\Qfrag(M)$. Then $\Delta^\epsrear$ is entirely
contained in $M$. Indeed, if a facet $F\in\Delta^\epsrear$ lies in $\Zrear$,
then $F$ is automatically in $M$. And if $F$ is not in $\Zrear$, then $F$ is
shared by $\Delta^\epsrear$ and $(\Delta')^\epsfr$ for another fragment
$\Delta'$. Hence $\Delta$ immediately precedes $\Delta'$, implying that
$\Delta'$ lies in the region between $M$ and $\Zrear$. Then $F$ is in $M$, as
required.

For $\Delta$ as above, the set $\Qfrag(M)-\{\Delta\}$ is again an ideal of
$(\Qfrag,\prec)$, and therefore it is expressed as $\Qfrag(M')$ for a
w-membrane $M'$. Moreover, $M'$ is obtained from $M$ by replacing the disk
$\Delta^\epsrear$ by $\Delta^\epsfr$. We call the transformation $M\mapsto M'$
the \emph{lowering flip} in $M$ \emph{using} $\Delta$, and call the reverse
transformation $M'\mapsto M$ the \emph{raising flip} in $M'$ \emph{using}
$\Delta$. As a result, we obtain the following nice property.
  \begin{corollary} \label{cor:seq_membranes}
Let $M$ be a w-membrane of a cubillage $Q$. Then there exists a sequence of
w-membranes $M_0,M_1,\ldots,M_k\in \Mscrw(Q)$ such that $M_0=\Zfr$, $M_k=M$,
and for $i=1,\ldots,k$, $M_i$ is obtained from $M_{i-1}$ by the raising flip
using some fragment in $\Qfrag$.
  \end{corollary}


 \subsection{Weakly $r$-separated collections via w-membranes}  \label{ssec:wsepar-wmembr}
~Now we throughout assume that $r$ is odd and $d=r+2$. Consider a cubillage $Q$
on $Z=Z(n,d)$. Based on Theorem~\ref{tm:gen_witn} and
Corollary~\ref{cor:seq_membranes}, we establish the main result of
Sect.~\SEC{weaksep-cubil}.

 \begin{theorem} \label{tm:membr-wsepar}
For any w-membrane $M$ of a cubillage $Q$ on $Z(n,d)$, the set $V(M)$ of
vertices of $M$ (regarded as subsets of $[n]$) constitutes a maximal by size
weakly $r$-separated collection in $2^{[n]}$ (where, as before, $r$ is odd and
$d=r+2$). In particular, all w-membranes in $Q$ have the same number of
vertices, namely, $w_{n,d-2}$ ($=s_{n,d-2}$).
  \end{theorem}
  \begin{proof}
Let $M\in\Mscrw(Q)$ and consider a sequence $\Zfr=M_0,M_1,\ldots,M_k=M$ as in
Corollary~\ref{cor:seq_membranes}. Let $\Delta_1,\ldots,\Delta_k$ be the
fragments of $Q$ such that $M_i$ is obtained from $M_{i-1}$ by the raising flip
using $\Delta_i$. The collection $V(\Zfr)$ is weakly $r$-separated (as it is
strongly $r$-separated, cf.~\refeq{size_membr}), and our aim is to show that if
$V(M_{i-1})$ is weakly $r$-separated, then so is $V(M_i)$.

To show this, consider w-membranes $M,M'$ of $Q$ such that $M'$ is obtained
from $M$ by the raising flip using a fragment $\Delta\in\Qfrag$. Let
$\Delta=\Cfrag_h$ for a cube $C=(X\,|\,T=(p(1)<\ldots<p(d)))$ and $h\in[d]$. By
explanations in Sect.~\SEC{smembr}, $\Cfr$ and $\Crear$ differ by exactly two
vertices; namely, $V(\Cfr)=V(\Crim)\cup\{t_C\}$ and
$V(\Crear)=V(\Crim)\cup\{h_C\}$, where $t_C=Xp(2)p(4)\ldots p({d-1})$ and
$h_C=Xp(1)p(3)\ldots p(d)$ (cf.~\refeq{gencube_inner}). Define $R$ to be the
set of vertices of $\Crim$ occurring in $\Delta$, and let $r':=(d-1)/2$. We
consider three cases.
  \smallskip

\noindent\emph{Case 1}: ~$h\le r'$. Since the vertices of $\Delta$ are formed
by the sections $S_{h-1}(C)$ and $S_h(C)$,
   $$
  V(\Delta)=(X\,|\,\tbinom{T}{h-1})\cup(X\,|\,\tbinom{T}{h}) \quad\mbox{and}
  \quad R\subseteq V(\Deltafr)\cup V(\Deltarear)
  $$
(cf.~\refeq{hor_sect}). Also $V(\Deltafr)\subseteq V(\Delta^\epsfr)$ and
$V(\Deltarear)\subseteq V(\Delta^\epsrear)$. If $h<r'$, then all vertices of
$\Delta$ belong to $\Crim$; this implies
$V(\Delta^\epsfr)=R=V(\Delta^\epsrear)$. And if $h=r'$, then the only vertex of
$\Delta$ not in $R$ is $t_C$. Since $t_C\in V(\Cfr)$, $t_C$ belongs to
$\Delta^\epsfr$. But $t_C$ also lies in the upper facet $S_{r'}(C)$ (in view of
$|p(2)p(4)\ldots p(d-1)|=r'$), and this facet is included in $\Delta^\epsrear$.
Hence $t_C\in \Delta^\epsfr\cap \Delta^\epsrear$, implying
$V(\Delta^\epsfr)=V(\Delta^\epsrear)$.
  \smallskip

\noindent\emph{Case 2}: ~$h\ge r'+2$. This is ``symmetric'' to the previous
case. If $h>r'+2$, then all vertices of $\Delta$ belong to $\Crim$, implying
$V(\Delta^\epsfr)=R=V(\Delta^\epsrear)$. And if $h=r'+2$, then $\Delta^\epsfr$
includes the lower facet $S_{r'+1}(C)$, which in turn contains the vertex $h_C$
(since $|p(1)p(3)\ldots p(d)|=r'+1$). Also $h_C\in V(\Crear)$ implies $h_C\in
V(\Delta^\epsrear)$, and we again obtain $V(\Delta^\epsfr)=V(\Delta^\epsrear)$.
  \smallskip

Thus, in both cases the raising flip $M\mapsto M'$ using $\Delta$ does not
change the vertex set of the w-membrane in question.
 \smallskip

\noindent\emph{Case 3}: ~$h=r'+1$. This case is most important. Now the lower
facet $S_{h-1=r'}(C)$ of $\Delta$ contains $t_C$, while the upper facet
$S_{h=r'+1}(C)$ contains $h_C$. Hence $t_C\in V(\Delta^\epsfr)$ and $h_C\in
V(\Delta^\epsrear)$. On the other hand, neither $t_C$ belongs to
$\Delta^\epsrear$ ($=\Deltarear\cup S_{r'+1}(C)$), nor $h_C$ belongs to
$\Delta^\epsfr$ ($=\Deltafr\cup S_{r'}(C)$).

It follows that $V(\Delta^\epsrear)=(V(\Delta^\epsfr)-\{t_C\})\cup\{h_C\}$, and
therefore the raising flip $M\mapsto M'$ using $\Delta$ replaces $t_C$ by
$h_C$, while preserving the other vertices of the w-membrane. Note also that
the vertices of $\Delta$ different from $t_C$ and $h_C$ form just the collection of
sets $XS$ such that $S$ runs over $\Nscr(\tilde P,\tilde Q)$, the set of
neighbors of $\tilde P:=p(2)p(4)\ldots p(d-1)$ and $\tilde Q:=p(1)p(3)\ldots
p(d)$.

Now applying Theorem~\ref{tm:gen_witn} to $\Wscr:=V(M),\,X,\tilde P,\tilde Q$
and $U:=X\tilde P$, we conclude that $\Wscr(M')$ is weakly $r$-separated, as
required.

This completes the proof of the theorem.
\end{proof}

It should be noted that any w-membrane in a cubillage on $Z(n,3)$ can be
expressed as a \emph{quasi-combined tiling} in the planar zonogon $Z(n,2)$, and
in this particular case, the statement of Theorem~\ref{tm:membr-wsepar} with
$r=1$ is equivalent to Theorem~3.4 in~\cite{DKK1}.

Also Theorem~\ref{tm:membr-wsepar} together with~\refeq{w-lattice} implies the
following property of the set $\bfW^\ast_{n,r}$ of representable maximal by
size weakly $r$-separated collections in $2^{[n]}$ (defined in the
Introduction).
  \begin{corollary} \label{cor:posetWast}
~$\bfW^\ast_{n,r}$ is a poset with the unique minimal element $V(\Zfr(n,r+2))$
and the unique maximal element $V(\Zrear(n,r+2))$ in which any two neighboring
elements are linked by a (raising or lowering) combinatorial flip.
  \end{corollary}

A natural question is whether any two members of the set $\bfW^=_{n,r}$
(including $\bfW^\ast_{n,r}$) can be connected by a sequence of flips. This is strengthened in
the following
 \medskip

\noindent \textbf{Conjecture 1} ~Let $r$ be odd and $n>r+1$. Then any maximal
by size weakly $r$-separated collection $\Wscr\subseteq 2^{[n]}$ is
representable, i.e., there exists a cubillage $Q$ on $Z(n,r+2)$ and a
w-membrane $M$ of (the fragmentation of) $Q$ such that $V(M)=\Wscr$.
 \medskip

This together with Theorem~\ref{tm:membr-wsepar} would imply
$\bfW^\ast_{n,r}=\bfW^=_{n,r}$. The above assertion has been proved for $r=1$;
see Theorem~3.5 in~\cite{DKK1}.


\section{The non-purity phenomenon for the weak $r$-separation}  \label{sec:non-pure}

Suppose that $R$ is a symmetric binary relation on elements of a
set $N$ and let $G$ be the graph whose vertices are the elements of $N$ and
whose edges are the pairs $\{u,v\}$ of distinct vertices subject to $uRv$. Let $\Cscr$ be
the set of cliques in $G$ (where a \emph{clique} is meant to be an
inclusion-wise maximal subset of vertices of which any two are connected by
edge). Then $\Cscr$ is said to be \emph{pure} if all cliques of $G$ have the
same size.

Recall that for an odd $r$ and $n>r$, ~$\bfW_{n,r}$ denotes the set of all
maximal by inclusion weakly $r$-separated collections in $2^{[n]}$. It was shown
in~\cite{DKK0} that $\bfW_{n,1}$ is pure for any $n$ (which affirmatively
answers Leclerc-Zelevinsky's conjecture on maximal weakly separated set-systems
in~\cite{LZ}). In other words, $\bfW_{n,1}=\bfW^=_{n,1} $ (=$\bfW^\ast_{n,1}$).

In this section we show that $\bfW_{n,r}$ need not be pure when $n=6$
and $r=3$. In fact, we borrow a construction from~\cite[Sect.~3]{DKK3} where it
is used to demonstrate the non-purity behavior for strongly $3$-separated
set-systems. (Note that by a general result due to Galashin and
Postnikov~\cite{GP},
 \begin{numitem1} \label{eq:rnr}
the set $\bfS_{n,r'}$ of all inclusion-wise maximal strongly $r'$-separated collections in
$2^{[n]}$ is pure if and only if $\min\{r',n-r'\}\le 2$, among all integers
$r',n$;
  \end{numitem1}
so $(n,r')=(6,3)$ is the smallest case when the non-purity of $\bfS_{n,r'}$
happens.)

To construct a non-pure set-system of our interest, consider the zonotope
$Z=Z(6,4)$. Note that the set $V(Z)$ of vertices of (the boundary of) $Z$
consists of all intervals and all 2-intervals containing 1 or 6. (This
relies on two observations: (a) any $A\subseteq[6]$ is a vertex of some
cubillage in $Z$, and therefore $V(Z)\cup\{A\}$ is (strongly) 3-separated,
by~\refeq{cub-separ}; and (b) the intervals and the 2-intervals containing 1 or 6
are exactly those subsets of $[6]$ that are $r'$-interlaced with any subset of
$[6]$, where $r'\le 4$.)

A direct enumeration shows that $|V(Z)|=52$. Therefore, $2^6-52=12$
subsets of the set $[6]$ are not in $V(Z)$; these are:
  \begin{numitem1} \label{eq:12sets}
~24, 245, 25, 235, 35, 135, 1356, 136, 1346, 146, 1246, 246.
  \end{numitem1}
(Recall that $a\cdots b$ stands for $\{a,\ldots,b\}$.) Let $A_i$ denote $i$-th
member in this sequence (so $A_1=24$ and $A_{12}=246$). Form the collection
  $$
  \Ascr:=V(Z)\cup\{A_1,A_5,A_9\}.
  $$
It consists of $52+3=55$ sets, whereas the number $s_{6,3}=w_{6,3}$ is equal to
$\binom{6}{0}+\binom{6}{1}+ \binom{6}{2}+\binom{6}{3}+\binom{6}{4}=57$. Now the
non-purity of $\bfW_{6,3}$ is implied by the following
  \begin{lemma} \label{lm:Ascript}
$\Ascr$ is a maximal weakly 3-separated collection in $2^{[6]}$.
  \end{lemma}
  \begin{proof}
As mentioned above, any two sets $X\in V(Z)$ and $Y\in\Ascr$ are 3-separated, and
therefore they are weakly 3-separated. Observe that $|A_{i-1}\triangle A_i|=1$
for any $1\le i\le 12$ (where $A_0:=A_{12}$ and $A\triangle B$ stands for
$(A-B)\cup(B-A)$). Then any $A,A'\in\{A_1,A_5,A_9\}$ satisfy $|A\triangle
A'|\le 4$. This implies that $A$ and $A'$ are 3-separated. Therefore,
the collection $\Ascr$ is weakly 3-separated.

The maximality of $\Ascr$ follows from the observation that adding to $\Ascr$
any member of $\{A_i: 1\le i\le 12,\; i\ne 1,5,9\}$ would violate the weak
3-separation. Indeed, a routine verification shows that $A_1$ is not weakly
3-separated from any of $A_6,A_7,A_8$, and similarly for $A_5$ and
$\{A_{10},A_{11},A_{12}\}$, and for $A_9$ and $\{A_2,A_3,A_4\}$.
 \end{proof}

\noindent \textbf{Remark 1.} ~To visualize a verification in the above proof, it
is convenient to use the circular diagram in Fig.~\ref{fig:circular} where the
sets from the sequence in~\refeq{12sets} are disposed in the cyclic order. Here
the sets $A_1,A_5,A_9$ are drawn in boxes and connected by lines with those sets
where the weak 3-separation is violated.
  \medskip

\begin{figure}[htb]
  \vspace{0.2cm}
\begin{center}
\includegraphics{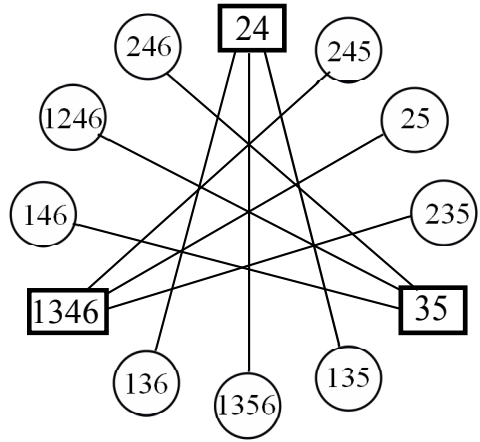}
\end{center}
 \caption{the circular diagram for $A_1,\ldots,A_{12}$.}
 \label{fig:circular}
  \end{figure}

In conclusion note that in light of the complete characterization for the strong
separation case in~\refeq{rnr}, it is tempting to characterize all pairs $(n,r)$
(with $r$ odd) for which $\bfW_{n,r}$ is pure. In particular, this is so when
$r=1$ (by~\cite{DKK0}), and it is not difficult to check that $\bfW_{n,r}$ is
pure if $n-r\le 2$. We conjecture that the remaining cases of $(n,r)$ give the
non-purity (similarly to~\refeq{rnr}), i.e., that for an odd $r$ and $n>r$,
$\bfW_{n,r}$ is pure if and only if $\min\{r,n-r\}\le 2$.


\appendix


\section{Proofs of two propositions on acyclicity}

 \noindent\textbf{Proof of Proposition~\ref{pr:acycl_all}}
~Let $C$ immediately precede $C'$, and let the cubes $C$, $C'$ and the facet
$F:=\Crear\cap(C')^{\rm fr}$ be of the form $(X\,|\,T)$, $(X'\,|\,T')$ and
$(\tilde X\,|\,\tilde T)$, respectively. Then $T=\tilde T\alpha$ and $T'=\tilde
T\beta$ for some $\alpha,\beta\in[n]$. Four cases are possible (as illustrated
in Fig.~\ref{fig:four}):
  \smallskip

(i) ~$X=X'=\tilde X$;

(ii) ~$X,X',\tilde X$ are different (then $\tilde X=X\alpha=X'\beta$);

(iii) ~$X\ne X'=\tilde X$ (then $\tilde X=X\alpha$);

(iv) ~$X'\ne X=\tilde X$ (then $\tilde X=X'\beta$).
 \smallskip

\begin{figure}[htb]
  \vspace{-0.2cm}
\begin{center}
\includegraphics[scale=1]{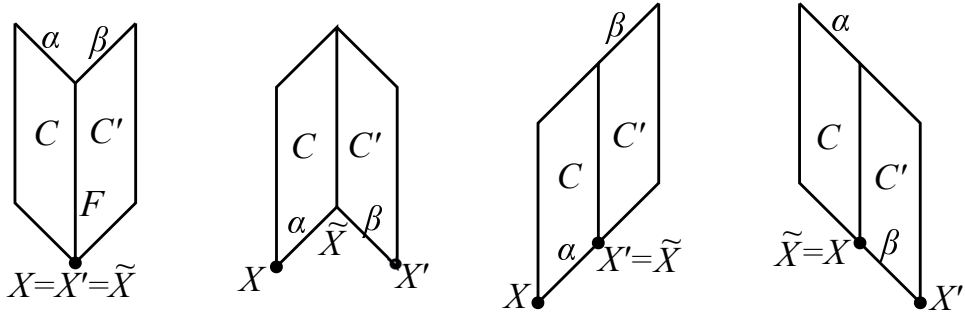}
\end{center}
\vspace{-0.3cm}
 \caption{Cases (i),(ii),(iii),(iv) (from left to right)}
 \label{fig:four}
  \end{figure}

Let us associate with a cube $C''=(X''\,|\,T'')$ a label
$\omega(C'')\in\{0,1,2\}$ by the following rule:
 \begin{itemize}
\item[($\ast$)] ~$\omega(C'')=0$\; if\; $n\ne X'',T''$; \quad
 $\omega(C'')=1$ \;if\; $n\in T''$; \quad
  $\omega(C'')=2$\; if\; $n\in X''$.
 \end{itemize}

The following observation is the key.
 \medskip

\noindent \textbf{Claim} ~\emph{For $C,C'$ as above, $\omega(C)\le\omega(C')$.}
 \medskip

\noindent\textbf{Proof of the Claim} ~ We may assume that
$\omega(C)\ne\omega(C')$. Then $n\in T\cup X\cup T'\cup X'$ but  $n$ belongs to neither $\tilde T$ nor $X\cap
X'$. This implies that either $\alpha =n$ or $\beta=n$ (in view of $\tilde
T=T-\alpha=T'-\beta$). Note that, as explained in Sect.~\SEC{smembr} (when proving~\refeq{unique_inner}),
  \begin{numitem1} \label{eq:Cfr_facets}
for a cube $C$, a facet $F_i(C)$ ($G_i(C)$) is in $\Cfr$ if and only if $d-i$
is even (resp. odd).
  \end{numitem1}
Using this for $C$ and $F$ as above and considering the inclusion
$F\subset\Crear$, one can conclude that if $\alpha=n$, then the root $\tilde X$
of $F$ and the root $X$ of $C$ are different (taking into account that $n$ is
the maximal element in $T$). In its turn, $F\subset (C')^{\rm fr}$ implies that
if $\beta=n$, then $\tilde X=X'$. This leads to the following:
  \begin{numitem1} \label{eq:alpha-beta-n}
$\alpha=n$ is possible only in cases (ii) and (iii), whereas $\beta=n$ is
possible only in cases~(i) and~(iii).
  \end{numitem1}
In particular, case~(iv) is impossible at all (when $\omega(C)\ne\omega(C')$).
As to the other three cases, we obtain from~\refeq{alpha-beta-n} that
  \begin{itemize}
\item[(a)] in case (i), $\omega(C)=0$ and $\omega(C')=1$ (since $n=\beta\in
T'$);
\item[(b)] in case (ii), $\omega(C)=1$ (since $n=\alpha\in T$) and
$\omega(C')=2$ (since $\tilde X=X\alpha=X'\beta$ implies $\alpha\in X'$);
\item[(c)] in case (iii),  if $\alpha=n$ then $\omega(C)=1$ and $\omega(C')=2$
(since $X'=X\alpha$), and if $\beta=n$ then $\omega(C)=0$ and $\omega(C')=1$.
  \end{itemize}

Thus, $\omega(C)\le\omega(C')$ holds in all cases, as required. \hfill\qed
\medskip

Now we finish the proof of the proposition by induction on $n$. This is trivial
when $n=d$, so assume that $n>d$ and that the assertion is valid for $(n',d')$
with $n'<n$.

Suppose, for a contradiction, that $\Gamma_{n,d}$ has a directed cycle
$\Cscr=(C_0,C_1,\ldots,C_k=C_0)$ (where each $C_i$ immediately precedes
$C_{i+1}$). Then the Claim implies that $\omega(C_i)$ is the same number $q$
for all $i$. Consider three cases (where $C_i=(X_i\,|\,T_i)$).
  \smallskip

\emph{Case 1}: $q=0$. Then $\Cscr$ is a directed cycle in $\Gamma_{n-1,d}$,
contrary to the inductive assumption.
  \smallskip

\emph{Case 2}: $q=2$. Define $X'_i:=X_i-n$ and $C'_i:=(X'_i\,|\,T_i)$,
$i=0,\ldots,k$. Then each $C'_i$ is a cube in $Z(n-1,d)$, and the sequence
$C'_0,C'_1,\ldots,C'_k$ forms a directed cycle in $\Gamma_{n-1,d}$; a
contradiction.
  \smallskip

\emph{Case 3}: $q=1$. Define $T'_i:=T_i-n$ and $C'_i:=(X_i\,|\,T'_i)$,
$i=0,\ldots,k$. Then each $C'_i$ can be regarded as a cube in $Z(n-1,d-1)$ (in
view of $|T'_i|=d-1$). Considering~\refeq{Cfr_facets} and using the fact that
$n$ is the maximal element in $T_i$, one can conclude that if $(Y\,|\,U)$ is a
facet in $C_i^{\rm fr}$, then $(Y\cap[n-1]\,|\, U\cap[n-1])$ is a facet in
$(C'_i)^{\rm rear}$, and similarly for facets in $C_i^{\rm rear}$ and
$(C'_i)^{\rm fr}$. Then the fact that $\Crear_i\cap\Cfr_{i+1}$ is a facet
implies that $(C'_i)^{\rm fr}\cap (C'_{i+1})^{\rm rear}$ is a facet as well. This
means that $C'_{i+1}$ immediately precedes $C'_i$. Therefore, the sequence
$C'_k,C'_{k-1},\ldots,C'_1,C'_0$ forms a directed cycle in $\Gamma_{n-1,d-1}$;
a contradiction.

This completes the proof of the proposition. \hfill\qed
  \medskip

  \noindent\textbf{Proof of Proposition~\ref{pr:acyc_Qfrag}}
~For a fragment $\Delta=\Cfrag_h$ of a cube $C=(X\,|\,T)$, denote $|X|+h-1/2$
by $\ell(\Delta)$, called the \emph{height} of $\Delta$.

Suppose that there exist fragments $\Delta_0,\Delta_1,\ldots,\Delta_k=\Delta_0$
forming a directed cycle in $\Gamma_{\Qfrag}$. Consider two consecutive
fragments $\Delta=\Delta_{i-1}$ and $\Delta'=\Delta_i$. Then the sides
$\Delta^\epsrear$ and $\Delta^\epsfr$ share a facet $F$, and either (a) $F$ is
a vertical facet of both (in terminology of~\refeq{fragm_cube}), or (b) $F$ is
the upper facet of $\Delta$ and the lower facet of $\Delta'$. Obviously,
$\ell(\Delta')=\ell(\Delta)$ in case~(a), and $\ell(\Delta')=\ell(\Delta)+1$ in
case~(b). This implies
  $$
 \ell(\Delta_0)\le\ell(\Delta_1)\le \cdots \le
\ell(\Delta_{k-1})\le\ell(\Delta_0).
  $$
Then all fragments $\Delta_i$ have the same height, and therefore each pair of
consecutive fragments shares a vertical facet. But this means that the sequence
of cubes containing these fragments forms a cycle in the graph $\Gamma_{n,d}$,
contrary to Proposition~\ref{pr:acycl_all}.
  \hfill\qed


\section{A concept of weak $r$-separation when $r$ is even}

Up to now, we have dealt with the weak $r$-separation when $r$ is odd. In this
section we attempt to introduce and explore an analogous concept when $r$ is
even.

For $A',B'\subseteq [n]$, we say that $A'$ \emph{surrounds $B'$ from the right}
if $\max(A'-B')>\max(B'-A')$.
\medskip

\noindent\textbf{Definition.} ~For an \emph{even} integer $r>0$ and an integer
$n>r$, sets $A,B\subseteq[n]$ are called \emph{weakly $r$-separated} if they
are $\tilde r$-interlaced with $\tilde r\le r+2$, and in case $\tilde r=r+2$,
either (a) $A$ surrounds $B$ from the right and $|A|\le|B|$, or (b) $B$
surrounds $A$ from the right and $|B|\le |A|$. Accordingly, a set-system
$\Wscr\subseteq 2^{[n]}$ is called weakly $r$-separated if any two members of
$\Wscr$ are such.
\medskip

(Note that this matches the definition for $r$ odd in the Introduction.)
\medskip

\noindent\textbf{Remark 2.} ~In contrast to the odd case, the size $|\Wscr|$ of
a weakly $r$-separated collection $\Wscr\subseteq 2^{[n]}$ with $r$ even can
exceed the value $s_{n,r}$ (defined in~\refeq{snr}). The simplest example is
given by $n=r+2$ and $\Wscr=2^{[n]}$. Indeed, in this case $s_{n,r}$ amounts to
$\sum(\binom{r+2}{i}\, \colon\, i=0,\ldots,r+1)=2^{r+2}-1$, which is less than
$|\Wscr|=2^{r+2}$. Observe that $\Wscr$ has only one pair $\{A,B\}$ of
$(r+2)$-interlaced sets, namely, $A=\{2,4,\ldots,r\}$ and
$B=\{1,3,\ldots,r-1\}$. These $A,B$ are weakly $r$-separated since $|A|=|B|$. (By
the way, one can see that this $\Wscr$ represents the vertex set of a
w-membrane in the fragmentation of the cube $C=(\emptyset\,|\, [n])$, forming
the trivial cubillage on $Z(n,n)$.)
  \medskip

Thus, if one wished to get rid of exceeding $s_{n,r}$, one should impose
additional restrictions on $\Wscr$. The above example prompts an idea on this
way.
\medskip

\noindent\textbf{Definition.} ~Let us say that a pair $\{A,B\}$ of subsets of
$[n]$ is a \emph{double $r$-comb} if $A,B$ are $(r+2)$-interlaced and
$|A\triangle B|=r+2$, i.e., $A\triangle B$ consists of elements $a_1<a_2<\ldots
<a_{r+2}$ of $[n]$, one of $A-B$ and $B-A$ is formed by $a_1,a_3,\ldots
a_{r+1}$, and the other by $a_2,a_4,\ldots,a_{r+2}$.
 \smallskip

A result involving a double $r$-comb and corresponding
neighbors is presented in the theorem below. This is in the spirit of
Theorem~\ref{tm:PQYstr} (concerning an odd $r$), but now the situation becomes more intricate.

More precisely, let $r':=r/2+1$, where $r$ is even as before. Let
$P=\{p_1,\ldots,p_{r'}\}$ and $Q=\{q_1,\ldots,q_{r'}\}$ consist of elements of
$[n]$ such that $p_1<q_1 <\ldots< p_{r'}<q_{r'}$. Define the sets
$\Nscr^\uparrow(P,Q)$ and $\Nscr^\downarrow(P,Q)$ of neighbors of $P,Q$ in the
same way as in~\refeq{NupP} and~\refeq{NdownQ} (where now $P,Q$ satisfy
$|P|=|Q|=r'$). Let $X\subseteq [n]-(P\cup Q)$ and let $Y\subseteq[n]$ be
different from $XP$ and $XQ$. We assert the following.

\begin{theorem} \label{tm:local_neighb}
Let $r,n,r',P,Q,X,Y$ be as above.

{\rm 1}. Suppose that $Y$ and $XS$ are weakly $r$-separated for any
$S\in\Nscr^\uparrow(P,Q)$, but $Y$ and $XP$ are not. If, in addition, $Y$ and
$XP$ are $(r+2)$-interlaced, then
  \begin{numitem1} \label{eq:XQaA}
 $Y=XQ\cup\{a\}$ for an element $a\in [n]-XPQ$ such that $a>p_1$.
  \end{numitem1}

{\rm 2}. Suppose that $Y$ and $XS$ are weakly $r$-separated for any
$S\in\Nscr^\downarrow(P,Q)$, but $Y$ and $XQ$ are not. If, in addition, $Y$ and
$XQ$ are $(r+2)$-interlaced, then
  \begin{numitem1} \label{eq:XPaB}
$Y=XP-b$ for some $b\in X$ such that $b>p_1$).
  \end{numitem1}
     \end{theorem}

(One can check that $Y$ as in~\refeq{XQaA} (resp.~\refeq{XPaB}) is indeed
weakly $r$-separated from any member of $\Nscr^\uparrow(P,Q)$ (resp.
$\Nscr^\downarrow(P,Q)$) but not from $XP$ (resp. $XQ$), and the essence of the
theorem is that there is no other $Y\ne XP,XQ$ with such a property.)
\medskip

\begin{proof}
Let us prove assertion 1.

Analyzing the proof of Lemma~\ref{lm:refinedPQ}, one can realize that it
remains valid for corresponding $P,Q,X,Y$ when $r$ is even as well. We further
rely on this lemma, borrowing, with a due care, terminology and notation from
Section~\SEC{proof2}. In particular: when sets $A,B$ are not weakly
$r$-separated, the pair $\{A,B\}$ is called \emph{bad}; $\Iscr$ abbreviates
$\Iscr(Y,XP)$; the intervals in $\Iscr$ are viewed as $\ldots
<A_{i-1}<B_i<A_i<B_{i+1}\ldots$\;, where $A_{i'}$ ($B_{i'}$) stands for a
$Y$-brick (resp. $XP$-brick). Also we may assume that $Y\cap X=\emptyset$
(though $Y$ and $P$ need not be disjoint).

Since $Y$ and $XP$ are required be $(r+2)$-interlaced, $\Iscr$ consists of $r'$ $Y$-bricks
and $r'$ $XP$-bricks. So we may assume that $\Iscr$ is viewed as either
  \begin{itemize}
\item[(V1)] ~$B_1<A_1<B_2<A_2<\ldots <B_{r'}<A_{r'}$, or
\item[(V2)] ~$A_1<B_2<A_2<\ldots <B_{r'}<A_{r'}<B_{r'+1}$.
  \end{itemize}

Next we consider two possible cases.
\medskip

\noindent\underline{\emph{Case A}}: $(\ast\ast)$ from Lemma~\ref{lm:refinedPQ}
is valid. Then $A_i=\{q_i\}$ for each $i$. Hence $Y-XP=Q$ and $|Y-XP|=r'\le
|XP-Y|$ (since each $XP$-brick $B_j$ contains an element of $XP-Y$). Suppose
that~(V1) takes place. Then $Y$-surrounds $XP$ from the right. This contradicts
the condition that $\{Y,XP\}$ is bad.

And if~(V2) takes place, then $p_1\notin XP-Y$ (since $p_1<q_1$ and
$\{q_1\}=A_1$). Hence $p_1\in Y$. For $S:=(P-p_1)q_1$, the transformation
$XP\mapsto XS$ replaces the first brick $\{q_1\}$ of $\Iscr$ by $\{p_1\}$,
forming the first $Y$-brick of $\Iscr(Y,XS)$. Then $|\Iscr(Y,XS)|=r'+2$,
$|XS|=|XP|$, $XS$ surrounds $Y$ from the right, and therefore the badness of
$\{Y,XP\}$ implies that of $\{Y,XS\}$; a contradiction.
\medskip

\noindent\underline{\emph{Case B}}: $(\ast)$ from Lemma~\ref{lm:refinedPQ} is
valid. Then $Y\cap P=\emptyset$, and the fact that $\Iscr$ has exactly $r'$
$XP$-bricks implies that $X=\emptyset$; so we may ignore $X$ in what follows.
If~(V2) takes place, then $XP$ surrounds $Y$ from the right. Since each
$P$-brick $B_i$ is a singleton, $|P|=r'\le|Y|$, contradicting the condition that
$\{Y,P\}$ is bad.

Now let~(V1) take place. Then $B_i=\{p_i\}$ for each $i$. Suppose that there is
$q_i\in Q$ such that $q_i\notin Y$. Then either~(a) $q_i$ lies in some
$Y$-brick $A_j$, or~(b) no brick of $\Iscr$ contains $q_i$. In case~(a), we
have $i=j$ (in view of $p_i<q_i<p_{i+1}$ and $p_i<A_i<p_{i+1}$, letting
$p_{r'+1}:=n+1$). Moreover, $\min(A_i)<q_i<\max(A_i)$ (since both
$\min(A_i)$ and $\max(A_i)$ are in $Y$). Taking $S:=Pq_i$, we obtain
$|\Iscr(Y,S)|=|\Iscr|+2$ (since $P\mapsto S$ replaces $A_i$ by the $S$-brick
$\{q_i\}$ and two $Y$-bricks). Hence $\{Y,S\}$ is bad; a contradiction.

In case~(b), three subcases are possible: either~(b1) $p_i<q_i<A_i$; or~(b2)
$A_i<q_i<p_{i+1}$, or~(b3) $i=r'$ and $A_{r'}<q_{r'}$. In these subcases,
taking as $S$ the neighbors $(P-p_i)q_i$, $(P-p_{i+1})q_i$, and $Pq_{r'}$,
respectively, one can see that $\{Y,S\}$ is bad.

Thus, $Q\subseteq Y$. Note that any $Y$-brick $A_i$ contains at most one
element of $Q$ (for if $A_i$ would contain $q_{j-1}$ and $q_j$ say, then $A_i$
should contain $p_j$ as well, which is impossible). It follows that each $A_i$
contains exactly one element of $Q$, namely, $q_i$. Since $\{Y,P\}$ is bad and
$Y$ surrounds $P$ from the right, there must be $|Y|>|P|=r'$. So at least one
$Y$-brick $A_i$ has size $\ge 2$. For such an $A_i$, taking $S:=Pq_i$, one can
see that $|\Iscr(Y,S)|=|\Iscr|$ and that $Y$ surrounds $S$ from the right. Then
$|Y|\le|S|$ (otherwise $\{Y,S\}$ is bad). This together with $|Y|>r'$ and
$|S|=|P|+1=r'+1$ gives $|Y|=r'+1$. The latter means that there is exactly one
brick $A_i$ of size $\ge 2$; moreover, $|A_i|=2$. Then $A_i=\{q_i,a\}$, where
$a$ is as required in~\refeq{XQaA}, yielding assertion~1 of the theorem.
  \smallskip

Assertion~2 of the theorem can be shown by symmetry and we leave details to the reader.
\end{proof}

\noindent\textbf{Remark 3.} Some neighbors of $P,Q$ arising in connection with
Theorem~\ref{tm:local_neighb} play an especial role. More precisely, let
$Y=XQ\cup\{a\}$ be as in~\refeq{XQaA}; then $p_i<a<p_{i+1}$ for some $i\in[r']$
(letting $p_{r'+1}:=n+1$). One can check that in the upper neighbor collection
$\{S\subset P\cup Q\,\colon\, S\ne P,Q,\; r'\le |S|\le r'+1\}$ (which includes
$\Nscrup(P,Q)$) there is \emph{exactly one} set $S$ such that $\{Y,XS\}$ is a
double $r$-comb; this is $S=Pq_i$. (Then $XS-Y=\{p_1,\ldots,p_{r'}\}$ and
$Y-XS=\{q_1,\ldots,q_{i-1},a,q_{i+1},\ldots, q_{r'}\}$.) Symmetrically, for
$Y=XP-b$ as in~\refeq{XPaB}, in the lower neighbor collection $\{S\subset P\cup
Q\,\colon\, S\ne P,Q,\; r'-1\le |S|\le r'\}$ (which includes $\Nscrdown(P,Q)$) there
is exactly one $S$ such that $\{Y,XS\}$ is a double $r$-comb. Namely, if
$p_i<b<p_{i+1}$ (letting $p_{r'+1}:=n+1$), then $S=Q-q_i$. (In this case,
$Y-XS=\{p_1,\ldots,p_{r'}\}$ and $XS-Y=\{q_1,\ldots,q_{i-1},b,q_{i+1},\ldots,
q_{r'}\}$). 
 \medskip

The rest of this section is devoted to a geometric construction representing a
class of $r$-separated collections. This relies on
Theorem~\ref{tm:local_neighb} and is in the spirit of the construction from
Sect.~\SSEC{wsepar-wmembr} (with $r$ odd), though looks a bit more intricate.
We will use terminology and notation from Sect.~\SEC{weaksep-cubil}.

As before, let $r$ be even and $r'=r/2+1$. For $d:=r+2$, consider a cubillage
$Q$ on the zonotope $Z=Z(n,d)$ and its fragmentation $\Qfrag$. For each cube
$C=(X|T)\in Q$, we distinguish two ``central'' fragments $\Cfrag_{r'}$ and
$\Cfrag_{r'+1}$. They share the middle horizontal section $S_{d/2}(C)$
($=C\cap H_{|X|+r'}$), which contains the specified vertices $t_C=XP$ and
$h_C=XQ$, where $T=(p_1<q_1<\ldots <p_{r'}<q_{r'})$, $P=\{p_1,\ldots,p_{r'}\}$
and $Q=\{q_1,\ldots,q_{r'}\}$ (so $\{t_C,h_C\}$ forms a double $r$-comb).
 \medskip

\noindent\textbf{Definitions.} For a cube $C\in Q$, the set $\Cfrag_{r'}\cup
\Cfrag_{r'+1}$ is called the \emph{doubled fragment}, or the \emph{center}, of
$C$ and denoted by $\Cfrag_{\rm cup}$; the remaining fragments $\Cfrag_h$ of $C$
($h\ne r',r'+1$) are called \emph{ordinary} ones. By the \emph{enlarged
fragmentation} of $Q$ we mean the complex generated by the centers and ordinary
fragments of all cubes of $Q$, denoted as $\Qfragen$, i.e., it is obtained from
$\Qfrag$ by merging the pieces $\Cfrag_{r'}$ and $\Cfrag_{r'+1}$ into
$\Cfrag_{\rm cup}$ for each $C\in Q$. Depending on the context, we may also think
of $\Qfragen$ as the collection of doubled and ordinary fragments over all the
cubes of $Q$.
 \medskip

This gives rise to an important subclass of w-membranes. More precisely, when a
w-membrane $M$ of $Q$ is a subcomplex (of dimension $d-1$) of $\Qfragen$, we
say that $M$ is an \emph{e-membrane}. It is not difficult to show that a
w-membrane $M$ of this sort is characterized by the property that no facet of
$M$ is the middle section of a cube of $Q$, or, equivalently, that for each
cube $C\in Q$, $M$ meets at most one vertex among $t_C,h_C$.

Like c- and w-membranes, the set of e-membranes of $Q$ forms a distribute
lattice. This is based on the following

  \begin{prop} \label{pr:acyc_Qfragen}
The directed graph $\Gamma_{\Qfragen}$ whose vertices are the fragments in
$\Qfragen$ and whose edges are the pairs $(\Delta,\Delta')$ of fragments such
that $\Delta$ immediately precedes $\Delta'$ (in the sense that
$\Delta^\epsrear$ and $(\Delta')^{\,\epsfr}$ share a facet) is acyclic.
\hfill\qed
  \end{prop}
  \begin{proof}
This is similar to the proof of Proposition~\ref{pr:acyc_Qfrag} and is briefly
as follows. Suppose that fragments $\Delta_0,\Delta_1,\ldots,\Delta_k=\Delta_0$
of $\Qfragen$ form a directed cycle in $\Gamma_{\Qfragen}$. For each $i$, let
$C_i$ be the cube of $Q$ containing $\Delta_i$. If $C_i=C_{i+1}$, then the
height of $\Delta_{i+1}$ is greater than that of $\Delta_i$. Therefore, a
maximal subsequence $S$ of different cubes among $C_0,C_1,\ldots,C_{k-1}$
consists of more than one element. Moreover, consecutive cubes in $S$ share a
(vertical) facet, whence $S$ determines a directed cycle in $\Gamma_{n,d}$,
contradicting Proposition~\ref{pr:acycl_all}.
  \end{proof}

Thus, the transitive closure of the above relation on the fragments of
$\Qfragen$ forms a partial order, denoted as $\precen$. As a consequence
(cf.~\refeq{w-lattice} and Corollary~\ref{cor:seq_membranes}):
   \begin{numitem1} \label{eq:e-lattice}
For an e-membrane $M$ of $Q$, let $\Qfragen(M)$ denote the collection of
fragments of $\Qfragen$ lying between $\Zfr$ and $M$. Then the set $\Mscre(Q)$
of e-membranes of a cubillage $Q$ on $Z(n,d)$ is a distributive lattice, with
the minimal element $\Zfr$ and the maximal element $\Zrear$, in which for
$M,M'\in\Mscre(Q)$, the e-membranes $M\wedge M'$ and $M\vee M'$ satisfy
$\Qfragen(M\wedge M')=\Qfragen(M)\cap \Qfragen(M')$ and $\Qfragen(M\vee
M')=\Qfragen(M)\cup \Qfragen(M')$.
  \end{numitem1}
  \begin{numitem1} \label{eq:seq_emembr}
Let $M$ be an e-membrane of $Q$. Then there exists a sequence of e-membranes
$M_0,M_1,\ldots,M_k\in \Mscre(Q)$ such that $M_0=\Zfr$, $M_k=M$, and for
$i=1,\ldots,k$, $M_{i-1}$ is obtained from $M_i$ by the lowering flip using
some maximal (w.r.t. $\precen$) fragment $\Delta$ in $\Qfragen(M_i)$ (in the
sense that $\Delta^\epsrear\subset M_i$, and $M_{i-1}$ is obtained from $M_i$ by
replacing the disk $\Delta^\epsrear$ by $\Delta^\epsfr$).
  \end{numitem1}

Based on the above properties, we obtain a geometric result which can be
viewed, to some extent, as a counterpart of Theorem~\ref{tm:membr-wsepar}
(concerning the odd case).

\begin{theorem} \label{tm:emembr}
Let $r$ be even and $d:=r+2$. Suppose that a cubillage $Q$ on $Z=Z(n,d)$
possesses the property that
 \begin{itemize}
 \item[\rm(P)] no e-membrane of $Q$ has a pair of vertices forming a double
 $r$-comb.
 \end{itemize}
Then for any e-membrane $M$ of $Q$,
  \begin{itemize}
\item[\rm(i)] the set $V(M)$ of vertices of $M$ (regarded as subsets of $[n]$)
is weakly $r$-separated;
 \item[\rm(ii)] $|V(M)|=s_{n,r}$.
 \end{itemize}
 \end{theorem}
 \begin{proof}
We argue as in the proof of Theorem~\ref{tm:membr-wsepar}. For $M\in\Mscre(Q)$,
consider a sequence $\Zfr=M_0,M_1,\ldots,M_k=M$ of e-membranes of $Q$ as
in~\refeq{seq_emembr}. Since $\Zfr$ satisfies (i),(ii), it suffices to prove
the following assertion.
 \begin{numitem1} \label{eq:emembrMMp}
For $M,M'\in\Mscre(Q)$, let $M'$ be obtained from $M$ by the raising flip using
a (double or ordinary) fragment $\Delta$ of $\Qfragen$, and suppose that $M$
satisfies~(i),(ii). Then $M'$ satisfies (i),(ii) as well.
  \end{numitem1}

To show this, assume that $\Delta$ belongs to a cube $C=(X|T)\in Q$. When
$\Delta$ is ordinary, i.e., $\Delta=\Cfrag_h$ with $h\le r'-1$ or $h\ge r'+2$
(where $r'=r/2+1$), then $V(M')=V(M)$, and we are done (cf. explanations in
Cases~1 and~2 of the proof of Theorem~\ref{tm:membr-wsepar}).

So let $\Delta$ be the center $\Cfrag_{\rm cup}$ of $C$. The raising flip using
$\Delta$ replaces in $M$ the side $\Delta^\epsfr$ by $\Delta^\epsrear$. One can
see that the vertex $t_C$ of $C$ is in $\Delta^\epsfr$ by not in
$\Delta^\epsrear$, while $h_C$ is in $\Delta^\epsrear$ by not in
$\Delta^\epsfr$, and that the other vertices of $\Delta$ and $\Delta'$
coincide. Therefore, the flip replaces $t_C=XP$ by $h_C=XQ$, yielding
$V(M')=(V(M)-\{t_C\})\cup\{h_C\}$, where $T=(p_1<q_1<\ldots<p_{r'}<q_{r'})$,
$P=\{p_1,\ldots,p_{r'}\}$ and $Q=\{q_1,\ldots,q_{r'}\}$.

We have $|V(M')|=|V(M)|$; so $M'$ satisfies~(ii). Suppose, for a contradiction,
that~(i) is false for $M'$, i.e., there are two vertices of $M'$ that are not
weakly $r$-separated from each other. Then one of them is $XQ$, and the other,
$Y$ say, belongs to $M$ and differs from $XP$. By~(i) for $M$, the vertex $Y$
is weakly $r$-separated from $XS$ for each neighbor $S\in\Nscrdown(P,Q)$. (Note
that $XS$ lies in $\Delta^\epsfr$ even if $|S|=r'$.) So we can apply
assertion~2 of Theorem~\ref{tm:local_neighb} and conclude that $Y$ is viewed as
in~\refeq{XPaB}. But then, as mentioned in Remark~3, there is $S\in
\Nscrdown(P,Q)$ such that $\{Y,XS\}$ is a double $r$-comb (namely, $Y=XP-b$ and
$S=Q-q_i$, where $p_i<b<p_{i+1}$). Hence $M$ contains a double $r$-comb,
contrary to condition~(P).

Thus, \refeq{emembrMMp} is valid, and the theorem follows.
 \end{proof}

\noindent\textbf{Remark 4.} ~By the construction of an e-membrane $M$ of a cubillage
$Q$, $M$ has no double $r$-comb of the form $\{t_C,h_C\}$ for a
cube $C$ of $Q$. However, a priori it is not clear whether $M$ is free of
double $r$-combs at all. We conjecture that this is so for any e-membrane, i.e.,
property~(P) holds for any cubillage $Q$. Its validity would give a
strengthening of Theorem~\ref{tm:emembr}. We state it as follows:
 \medskip

\noindent \textbf{Conjecture 2} ~ For $r$ even, the vertex set $V(M)$ of any
e-membrane $M$ of an arbitrary cubillage $Q$ on $Z(n,r+2)$ gives a weakly
$r$-separated collection.
 \medskip

\noindent(Note that such a $V(M)$ is automatically of size $s_{n,r}$, by
explanations above.) It is tempting to conjecture a sharper property (which is
just a direct analog of Theorem~\ref{tm:membr-wsepar}), claiming that $V(M)$ is
weakly $r$-separated for any \emph{w-membrane} of a cubillage $Q$ on $Z(n,r+2)$
(where $|V(M)|$ may exceed $s_{n,r}$), but we do not go so far at the moment.

We finish with an analogue of Conjecture~1:
\medskip

\noindent \textbf{Conjecture 3} ~ For $r$ even, the maximal size of a weakly
$r$-separated collection $W\subseteq 2^{[n]}$ without double $r$-combs is equal
to $s_{n,r}$ and such a $W$ with $|W|=s_{n,r}$ is representable, in the sense
that there exists a cubillage $Q$ on $Z(n,r+2)$ and an e-membrane $M$ of $Q$
such that $V(M)=W$.

\end{document}